%% file: ms.tex
\documentclass[onefignum,onetabnum]{siamart190516}

\input{ex_shared}

\ifpdf
\hypersetup{
  pdftitle={High Order Numerical Homogenization for Dissipative 
  Ordinary Differential Equations},
  pdfauthor={Zeyu Jin, and Ruo Li}
}
\fi

\usepackage{geometry}
\begin{document}

\maketitle

\input{article_abstract}

\input{article_introduction}

\input{article_preliminaries.tex}

\input{article_model.tex}

\input{article_algorithm.tex}

\input{article_analysis.tex}

\input{article_numerical.tex}

\input{article_conclusion.tex}

\input{article_appendix.tex}

\bibliographystyle{siamplain}
\bibliography{ref}

\end{document}


\maketitle

\section{A detailed example}

Here we include some equations and theorem-like environments to show
how these are labeled in a supplement and can be referenced from the
main text.

\bibliographystyle{siamplain}
\bibliography{ref}

%% file: ex_shared.tex
\usepackage{algorithmicx}
\usepackage{algpseudocode}
\usepackage{lipsum}
\usepackage{amssymb}
\usepackage{amsmath}
\usepackage{amsfonts}
\usepackage{graphicx}
\usepackage{epstopdf}

\usepackage{amsopn}
\usepackage{microtype}
\usepackage{subfigure}

\ifpdf
  \DeclareGraphicsExtensions{.eps,.pdf,.png,.jpg}
\else
  \DeclareGraphicsExtensions{.eps}
\fi

\headers{High Order Numerical Homogenization}{Zeyu Jin and Ruo Li}

\title{High Order Numerical Homogenization for Dissipative 
  Ordinary Differential Equations}

\author{Zeyu Jin\thanks{Yuanpei College \& School of Mathematical 
Sciences, Peking University
    (\email{jinzy@pku.edu.cn}).} \and Ruo Li\thanks{CAPT, LMAM \&
    School of Mathematical Sciences, Peking University
    (\email{rli@math.pku.edu.cn}).}}

\newsiamremark{assumption}{Assumption}
\newsiamremark{example}{Example}
\newsiamremark{remark}{Remark}
\newsiamremark{experiment}{Experiment}
\newsiamremark{hypothesis}{Hypothesis}
\crefname{hypothesis}{Hypothesis}{Hypotheses}
\crefname{assumption}{Assumption}{Assumptions}
\newsiamthm{claim}{Claim}

\newcommand{\mr}{\mathbb{R}}
\newcommand{\rx}{\mathbb{R}^{n_x}}
\newcommand{\ry}{\mathbb{R}^{n_y}}
\newcommand{\eps}{\varepsilon}
\newcommand{\rd}{{\mathrm{d}}}
\newcommand{\rr}{{\mathrm{r}}}
\newcommand{\od}[2]{\dfrac{\mathrm{d} #1}{\mathrm{d} #2}}
\newcommand{\pd}[2]{\dfrac{\partial #1}{\partial #2}}
\newcommand{\ms}{\mathcal{S}}
\newcommand{\ip}[2]{\left\langle{#1},{#2}\right\rangle} 
\newcommand{\mc}{\mathcal{C}}
\newcommand{\ml}{\mathcal{L}}
\newcommand{\hg}{\hat{\Gamma}}
\newcommand{\hgd}{\hat{\Gamma}^\mathrm{d}}
\newcommand{\tx}{\tilde{x}}
\newcommand{\ty}{\tilde{y}}

\newcommand{\norm}[1]{\left\|#1\right\|}
\newcommand{\e}[2]{{#1 \times 10^{#2}}}
\graphicspath{{fig/}}


%% file: article_abstract.tex
\begin{abstract}
  We propose a high order numerical homogenization method for
  dissipative ordinary differential equations (ODEs) containing two
  time scales. Essentially, only first order homogenized model
  globally in time can be derived. To achieve a high order method, we
  have to adopt a numerical approach in the framework of the
  heterogeneous multiscale method (HMM). By a successively refined
  microscopic solver, the accuracy improvement up to arbitrary order
  is attained providing input data smooth enough. Based on the
  formulation of the high order microscopic solver we derived, an
  iterative formula to calculate the microscopic solver is then
  proposed. Using the iterative formula, we develop an implementation
  to the method in an efficient way for practical
  applications. Several numerical examples are presented to validate
  the new models and numerical methods.
\end{abstract}

\begin{keywords}
  Multiscale methods; Dissipative systems; Homogenization; Correction 
  model
\end{keywords}

\begin{AMS}
  34E13, 65L04
\end{AMS}

%% file: article_introduction.tex
\section{Introduction}
Multiple-time-scale problems are often encountered in many disciplines such as chemical kinetics \cite{Macnamara2008Multiscale}, molecular dynamics \cite{Car1985Unified,Zhang1998A} and celestial mechanics \cite{Laskar1994Large}. 
There are many studies concerning stiff systems of ordinary differential equations (ODEs), especially those with two time scales. 
In general, these systems can be divided into two categories \cite{Weinan2003Analysis}. One is dissipative systems where fast variables tend to the stationary state at an exponential rate. The other one is oscillatory systems where fast variables oscillate in some orbits. 

It is impossible to resolve all the time scales and capture all the variables 
numerically due to limited computing power. In many problems, we are only 
interested in the dynamics of slow macroscopic variables. There is some work 
on designing efficient numerical algorithms for stiff ODEs, such as implicit 
Runge-Kutta methods \cite{Hairer1980Solving}, backward differentiation 
formulas \cite{Brayton1972A}, Rosenbrock methods \cite{Kaps1985Rosenbrock} and 
projective methods \cite{Gear2001Projective}.

It is well known \cite{Papanicolaou1976Some} that, as 
$\eps\rightarrow0$, the dynamics of slow variables satisfy a limiting 
equation, which can be obtained by averaging methods. For simple 
systems, the limiting equation can be derived by analytical tools. 
For complex systems, however, we have to sample the fast variables to 
approximate the limiting equation. A famous method of this type is 
the heterogeneous multiscale method 
\cite{Weinan2012Review,Weinan2012The,Engquist2005HMMODEs}. In HMM, 
there is a microscopic solver to sample the fast variables, and a 
macroscopic solver to evolve the slow variables. Some results of 
numerical analysis on this method can be found in 
\cite{Weinan2003Analysis}. There are three main sources of errors of 
HMM, including modeling error, sampling error and truncation error of 
the macroscopic solver. The modeling error is of order 
$\mathcal{O}(\eps)$. When $\eps$ is sufficiently small, the modeling 
error is small enough. However, when $\eps$ is relatively but not 
extremely small, the modeling error is not ignorable. Therefore, it 
is necessary to propose a correction model to reduce the modeling 
error in this situation.

There have been several attempts to reduce modeling error in multiscale method. For example, in \cite{Chartier2012Higher}, they use B-series to derive a high-order stroboscopic averaged equations for a kind of highly oscillatory systems.
In \cite{Jiang2015A,Wu}, they develop a first order correction model for sediment 
transport in sub-critical case, where the modeling error can be reduced to 
$\mathcal{O}(\eps^2)$. 

In this paper, we develop a novel high-order correction model and the 
corresponding numerical algorithms for stiff dissipative system of 
ODEs. Our studies begin with the theory of invariant manifold 
for ODEs \cite{Givon2004Extracting,pavliotis2008multiscale}. It can 
be proved that the invariant manifold is a global attractor. 
Actually, the microscopic solver of HMM is designed to find an 
approximation to the invariant manifold with an error of 
$\mathcal{O}(\eps)$. In order to derive the correction models, we 
need to find high-order approximations to the invariant manifold.
By the asymptotic approximation method, the analytical expressions of the first several terms in the formal expansion can be obtained.
We can prove that this approximation shares similar properties to the 
invariant manifold. 
In other words, the trajectories tend to this approximation to 
invariant manifold at an exponential rate.
Once the trajectories are close to this approximation, then they can be approximated by a reduced model over a finite time horizon.
However, the terms obtained by the formal expansion are too 
complicated to be implemented, especially for higher-order methods. 
In addition, it requires the evaluation of high-order derivatives.
We design an iterative formula for the sake of practicality.
We can prove that the iterative method matches with the formal expansion in some sense. We can also prove that the approximation accuracy reaches $\mathcal{O}(\eps^{k+1})$ after $k$ iteration steps.
By the technology of numerical derivatives, we design a recursion method using the iterative scheme. We also present some numerical analysis on our algorithms.

The rest of this paper is arranged as follows. In \Cref{sec:pre}, we 
briefly introduce the multiscale dissipative systems 
and the heterogeneous multiscale methods. In \Cref{sec:model}, 
we present our models for high-order homogenization and analyze their 
properties. In \Cref{sec:alg}, we develop two types of 
algorithms in the framework of HMM. Some numerical analysis is 
presented in \Cref{sec:analysis}. Numerical results are shown in 
\Cref{sec:num} and the paper ends with a brief summary and 
conclusion in \Cref{sec:concl}.

%% file: article_preliminaries.tex
\section{Preliminaries}
\label{sec:pre}
\subsection{Dissipative systems}
\label{sec:pre_dissipative}
Let us consider the following ODEs with scale separation 
\cite{pavliotis2008multiscale}:
\begin{equation}
  \label{equation}
  \left\{
    \begin{aligned}
      & \od{x}{t} = f(x,y), \\
      & \od{y}{t} = \frac{1}{\eps}  g(x,y), \\
      & x|_{t = 0} = x_0, \ y|_{t = 0} = y_0,
    \end{aligned}
  \right.
\end{equation}
where $x \in \rx$ is the slow variable and $y \in \ry$ is the fast
variable, $n_x$ and $n_y$ are the dimension of $x$ and $y$,
respectively. The parameter $0< \eps \le \eps_0 \ll 1$
characterizes the separation of time scales.
Let $y^x$ be the solution of
\begin{equation}\label{equ:yx}
  \left\{
  \begin{aligned}
    & \od{y^x}{t} = \frac{1}{\eps}  g(x,y^x), \\
    & y^x|_{t = 0} = y_0,
  \end{aligned}
  \right.
\end{equation}
for any $x$ fixed. Suppose that $\rd \mu^x(y)$ is the corresponding 
invariant measure of \cref{equ:yx} satisfying that for any $\phi \in 
\ml^1(\ry; \rd \mu^x)$, 
\begin{displaymath}
  \lim\limits_{T \rightarrow +\infty} \frac{1}{T} \int_{0}^{T} 
  \phi(y^x(t)) \, \rd t = \int_{\ry} \phi(y) \, \rd \mu^x(y), \ 
  \text{ for } \mu^x-\text{a.e. } y_0 \in \ry.
\end{displaymath}
Let
\begin{displaymath}
  F(x) = \lim\limits_{T\rightarrow+\infty} \frac{1}{T} \int_{0}^{T} 
  f(x,y^x(t)) \, \rd t = \int_{\ry} f(x,y) \, \rd \mu^x(y).
\end{displaymath}
Under some appropriate assumptions \cite{Papanicolaou1976Some}, the
trajectory of $x(t)$ tends to a solution of the limiting equation
\begin{equation}\label{equ:limiting_equation}
  \od{X_0}{t}=F(X_0),
\end{equation}
as $\eps\rightarrow0$ in some sense.

Since now on, we assume the regularity of the functions $f$ and $g$. 
Precisely, we assume that
\begin{assumption}
  \label{assump:bound}
  The functions $f$ and $g$ are sufficiently smooth. In addition, 
  there exists $K \in \mathbb{N} \setminus \{0\}$ such that
  \footnote{ 
  Here we adopt standard notations of Sobolev spaces. The 
  space $W^{k,\infty}(\mathbb{R}^n,\mathbb{R}^m)$ is equipped with 
  the norm defined by
    \begin{displaymath}
      \norm{ \xi }_{ W^{k,\infty} (\mr^n, \mr^m) } := \max_{j = 0, 1, 
      \ldots, k} \norm{ \nabla^j \xi }_{\ml^\infty} , \quad \xi \in 
      W^{k,\infty} (\mr^n,\mr^m).
    \end{displaymath}
    When no ambiguity is possible, $\norm{ \cdot }_{ W^{k,\infty} 
    (\mr^n, \mr^m) }$ and $ W^{k,\infty} (\mr^n, \mr^m)$ are 
    abbreviated as $\norm{ \cdot }_{k,\infty}$ and $W^{k,\infty}$, 
    respectively. Let $|\cdot|_{k,\infty}$ defined by 
    \begin{displaymath}
      |\xi|_{k,\infty} := \norm{\nabla^k \xi}_{\ml^\infty}
    \end{displaymath}
    be the Sobolev semi-norm.}
  \begin{displaymath}
    f \in W^{K,\infty}(\rx \times \ry, \rx), \quad \nabla g \in
    W^{K,\infty}(\rx \times \ry, \mr^{n_y \times (n_x + n_y)}).
  \end{displaymath}
\end{assumption}
Furthermore, we always have the following assumption for $g$.
\begin{assumption}
  \label{assump:dissipative}
  For each $x \in \rx$, there exists $\gamma(x) \in \ry$ such that
  \begin{equation} \label{equ:assump_dissipative_1}
    g(x, \gamma(x)) = 0.
  \end{equation}
  In addition, there exists $\beta > 0$ such that
  \begin{equation}
    \label{equ:assump_dissipative_2}
    \ip{g(x,y) - g(x,\tilde{y})}{y - \tilde{y}} 
    \le -\beta |y-\tilde{y}|^2, 
    \quad \forall x \in \rx, \forall y, \tilde{y} \in \ry.
  \end{equation}
\end{assumption}
By \cref{assump:dissipative}, one can see that the dynamic
for $y$ with $x$ fixed has a unique, globally exponentially 
attracting point \cite{pavliotis2008multiscale}. At this time, the 
system \cref{equation} is called a \textit{dissipative system}. It is 
clear that the invariant measure $\rd \mu^x(y)$ at fixed $x$ 
is a one-point distribution at $y = \gamma(x)$. In addition, one has 
that $F(x) = f(x, \gamma(x))$. It was pointed out in 
\cite{pavliotis2008multiscale} that, under some appropriate 
assumptions, the modeling error between the solutions of 
\cref{equation} and \cref{equ:limiting_equation} is of order 
$\mathcal{O}(\eps)$ in a finite time horizon.
\begin{remark}
  \label{remark:monotone}
  Let $\tilde{g}_x = g(x, \cdot) : \ry \rightarrow \ry$.
  If $g(x,y)$ satisfies \cref{equ:assump_dissipative_2}, then we say
  that $\tilde{g}_x$ is a $\beta$-strongly dissipative operator or
  $-\tilde{g}_x$ is a $\beta$-strongly monotonic operator on $\ry$ by
  the definitions in \cite{Ryu2016A, Brezis2010Functional}. It can be
  deduced \cite{Songnian2018An} that, if $\tilde{g}_x$ is 
  $L$-Lipschitz continuous, and $\beta$-strongly dissipative, then 
  $\tilde{g}_x$ is a bijection on $\ry$, $\tilde{g}_x^{-1} : \ry 
  \rightarrow \ry$ is $\frac{1}{\beta}$-Lipschitz continuous, and 
  $\frac{\beta}{L^2}$-strongly dissipative.
\end{remark}

\subsection{Heterogeneous multiscale method}
\label{sec:pre_hmm}
The heterogeneous multiscale method (HMM) is a general strategy for 
multiscale problems \cite{Weinan2012The,Weinan2003Analysis}. It makes 
use of two solvers: a macroscopic solver and a microscopic solver. 
Let us consider 
the system \cref{equation}. As a macroscopic solver, a conventional 
explicit ODE solver is chosen to evolve 
\cref{equ:limiting_equation}. 
For example, take the forward Euler method with a time step 
$\Delta t$ as the macroscopic solver, which can be expressed as
\begin{displaymath}
x_{n+1} = x_n + \Delta t F(x_n).
\end{displaymath}
To evaluate $F(x_n)$, a microscopic solver is chosen to resolve the 
microscopic scale. In the case of the forward Euler method with a 
time step $\delta t$ , one gets that 
\begin{subequations} \label{equ:hmm_micro}
\begin{align}
\label{equ:hmm_micro_1}
& y_{n,m+1} = y_{n,m} + \frac{\delta t}{\eps} g(x_n,y_{n,m}), \ m = 
0, 1, \ldots, M-1, \\
\label{equ:hmm_micro_2}
& y_{n,0} \text{ suitably chosen.}
\end{align}
\end{subequations}
Then one can estimate $F(x_n)$ by
$$ F(x_n) \approx \sum_{m=0}^{M} K_{m,M} f(x_n,y_{n,m}), $$
where the weights $\{K_{m,M}\}$ should satisfy the constraint
$ \sum_{m=0}^{M} K_{m,M} = 1. $
As suggested in \cite{Weinan2003Analysis}, for dissipative systems, 
one can choose the weights as
\begin{displaymath}
\begin{aligned}
&K_{m,M} = 0, \ m = 0, 1, \ldots, M-1, \\
&K_{M,M} = 1.
\end{aligned}
\end{displaymath}
In other words, one may estimate that $F(x_n) = f(x_n, \gamma(x_n)) 
\approx f(x_n,y_{n,M}) $. In this situation, the microscopic solver 
\cref{equ:hmm_micro} can be regarded as a nonlinear solver for the 
equation $g(x_n,y) = 0$ with respect to $y$.


%% file: article_model.tex
\section{Models}
\label{sec:model}
In this section, we present our high-order correction models of the
limiting equation \cref{equ:limiting_equation} and analyze their
properties. We begin with asymptotic approximation to the invariant
manifold and then design an iterative formula to generate high-order
correction models automatically. The modeling error can be reduced to 
$\mathcal{O}(\eps^{k+1})$ in the $k$th-order correction model. 

\subsection{Invariant manifold}
\label{sec:model_attractor}
To begin with, we introduce the concept of invariant manifold 
\cite{pavliotis2008multiscale}. Assuming that there exists an 
invariant set $\ms_\eps$ of \cref{equation}, and that it can be 
represented as a smooth graph over $x$, namely, there exists a 
function $\Gamma: \rx \times (0,\eps_0] \rightarrow \ry$ that is 
differentiable with respect to $x$, such that
\begin{displaymath}
\ms_\eps = \left\{
  (x,y): y = \Gamma(x,\eps),\ x \in \rx 
  \right\}, \quad \forall \eps \in (0,\eps_0].
\end{displaymath}
This implies that 
\begin{equation}
\label{equ:center}
g(x,\Gamma(x,\eps)) = \eps \nabla_x\Gamma(x,\eps) f(x,\Gamma(x,\eps)),
\ \forall x \in \rx.
\end{equation}
The equation \cref{equ:center} plays a central role in our 
discussions. It can be proved under some appropriate conditions 
that $\ms_\eps$ is a \textit{global 
attractor} of the system \cref{equation}, that is, for any initial 
values,
\begin{equation}
\label{equ:attractor}
\lim \limits_{t \rightarrow +\infty} |y(t)-\Gamma(x(t),\eps)| = 0.
\end{equation}
\begin{proposition}
  \label{prop:attractor}
  Suppose that $\nabla_y f(x,y)$ and $\nabla_x\Gamma(x,\eps)$ are 
  bounded, then $\ms_\eps$ is a global attractor of the system 
  \cref{equation} for sufficiently small $\eps$.
\end{proposition}
\begin{proof}
  Assuming that $\big| \nabla_y f(x,y) \big| \le C$ and 
  $\big| \nabla_x \Gamma(x,\eps) \big| \le C$.
  Let $z(t) = y(t) - \Gamma(x(t),\eps)$.
  By \cref{assump:dissipative}, one gets that 
  \begin{align*}
  	&\frac{1}{2} \od{|z|^2}{t} = \ip{z}{\od{z}{t}}
  	= \ip{y - \Gamma(x,\eps)}
  	  {\frac{1}{\eps} g(x,y) - \nabla_x \Gamma(x,\eps) f(x,y)}\\
  	=& \ip{y-\Gamma(x,\eps)}{\frac{1}{\eps} g(x,y) - \frac{1}{\eps}  
  	g(x,\Gamma(x,\eps))} \\
    & \qquad\quad + \ip{y - \Gamma(x,\eps)}
  	  {\nabla_x \Gamma(x,\eps) f(x,\Gamma(x,\eps)) - \nabla_x \Gamma(x,\eps) f(x,y)} \\
  	\le& \left(-\frac{\beta}{\eps} + C^2\right) |z|^2.
  \end{align*}
  This implies that 
  $|z(t)|^2 \le e^{-\frac{\beta}{\eps} t} |z(0)|^2$ when 
  $0< \eps \le \min \left\{ \frac{\beta}{2C^2}, \eps_0 \right\}$, and 
  then the proof is completed.
\end{proof}
\cref{prop:attractor} says that the trajectory of 
$(x(t),y(t))$ tends to the set $\ms_\eps$ at an exponential rate. 
Later we refer to this type of properties as the 
\textit{attractive property}.
Since $\ms_\eps$ is invariant, one obtains that, if the initial 
values lie on $\ms_\eps$, then the trajectory of $(x(t),y(t))$ 
stays on $\ms_\eps$ for any $t>0$. At this time, one may use 
\begin{equation}
  \label{equ:invariant}
  \od{X}{t} = f(X,\Gamma(X,\eps))
\end{equation}
rather than \cref{equation} to 
calculate the trajectory of $x(t)$. In other words, the system 
\cref{equation} can be decoupled. We refer to this type of 
properties as the \textit{decouplable property}. It should be 
emphasized that \cref{prop:attractor} depends on the existence of 
$\Gamma(x,\eps)$, however, the rest of this article does not 
depend on it.

As an example, we study the invariant manifold of the linearized 
equation of the system \cref{equation}.
\begin{example}
  \label{example:linear_1}
  We consider the following linear equation as a special case of the 
  system \cref{equation}:
  \begin{equation}
    \label{equ:linear}
    \left\{
    \begin{aligned}
      &\od{x}{t} = A_{11}x + A_{12}y + b_1,\\
      &\od{y}{t} = \frac{1}{\eps} (A_{21}x + A_{22}y + b_2),
    \end{aligned}
  \right.
  \end{equation}
  where 
  $A_{11} \in \mathbb{R}^{n_x\times n_x}$, 
  $A_{12} \in \mathbb{R}^{n_x\times n_y}$, 
  $A_{21} \in \mathbb{R}^{n_y\times n_x}$, 
  $A_{22} \in \mathbb{R}^{n_y\times n_y}$, 
  $b_1 \in \rx$, $b_2 \in \ry$. 
  We assume that $A_{22}$ is a negative definite matrix to satisfy 
  \cref{assump:dissipative}. 
  
  Suppose that the function $\Gamma(x,\eps)$ has a form of 
  $\Gamma(x,\eps) = Cx + d$, where 
  $C = C(\eps) \in \mathbb{R}^{n_y \times n_x}$ and 
  $d = d(\eps) \in \ry$.
  By \cref{equ:center}, one gets the following equations:
  \begin{subequations}
    \label{equ:center_linear}
    \begin{align}
      \label{equ:center_linear_1}
      &\eps C A_{12} C + \eps C A_{11} - A_{22} C - A_{21} = 0,\\
      \label{equ:center_linear_2}
      &(A_{22} - \eps C A_{12}) d = \eps C b_1 - b_2.
    \end{align}
  \end{subequations}
  It is known that \cref{equ:center_linear_1} is an algebraic matrix 
  Riccati equation \cite{freiling2002survey,Su1992The}. We can prove 
  the well-posedness of \cref{equ:center_linear} when $\eps$ is 
  sufficiently small. See \cref{thm:linear} in \cref{app:linear}.
\end{example}

\subsection{Asymptotic approximation}
\label{sec:model_expansion}
In this part, we consider the asymptotic approximation to 
$\Gamma(x,\eps)$. We calculate the first several terms in the formal 
asymptotic expansion, and then prove the attractive property and 
the 
decouplable property of this approximation.

\subsubsection{Formal expansions}
Suppose that $\Gamma(x,\eps)$ satisfies \cref{equ:center}. 
We try formal expansions of the form
\begin{subequations}
  \label{equ:expansion}
  \begin{align}
    \label{equ:expansion_1}
      &\Gamma(x,\eps) = \gamma_0(x) + \eps \gamma_1(x) 
      + \eps^2 \gamma_2(x) + \mathcal{O}(\eps^3),\\
    \label{equ:expansion_2}
      &\nabla_x \Gamma(x,\eps) = \nabla \gamma_0 (x) 
      + \eps \nabla \gamma_1(x) + \eps^2 \nabla \gamma_2(x) 
      + \mathcal{O}(\eps^3).
  \end{align}
\end{subequations}
By \cref{assump:dissipative}, we specify that 
$\gamma_0(x)=\gamma(x)$. By Taylor's expansion of $g(x,y)$ at 
$(x,\gamma(x))$ and the formal expansion in \cref{equ:expansion_1}, 
one notices that the left-hand side of \cref{equ:center} can be 
written as 
\begin{equation}
  \label{equ:expansion_left}
  \begin{aligned}
    & g(x,\Gamma(x,\eps))\\
    = & g(x,\gamma(x) + \eps \gamma_1(x) 
    + \eps^2 \gamma_2(x) + \mathcal{O}(\eps^3)) \\
    = & \eps G_y(x) \gamma_1(x) + \eps^2 \left( G_y(x) \gamma_2(x) 
    + \frac{1}{2} \left[ \sum_{j,k=1}^{n_y} G_{yy}^{i,j,k} 
    \gamma_1^j(x) \gamma_1^k(x) \right]_{i=1}^{n_y} \right) + 
    \mathcal{O}(\eps^3),
\end{aligned}
\end{equation}
where $G_y(x) = [G_y^{i,j}(x)]_{i,j=1}^{n_y} = 
\left[ \frac{\partial g^i}{\partial y^j} (x,\gamma(x))
  \right]_{i,j=1}^{n_y}$ and 
$G_{yy}(x) = [G_{yy}^{i,j,k}(x)]_{i,j,k=1}^{n_y} = 
\left[ \frac{\partial^2 g^i}{\partial y^j \partial y^k} (x,\gamma(x)) 
  \right]_{i,j,k=1}^{n_y}$. 
Similarly, by Taylor's expansion of $f(x,y)$ at $(x,\gamma(x))$ and 
the formal expansions in \cref{equ:expansion}, one obtains that the 
right-hand side of \cref{equ:center} can be written as
\begin{equation}
  \label{equ:expansion_right}
  \begin{aligned}
    & \eps \nabla_x \Gamma(x,\eps) f(x,\Gamma(x,\eps)) \\
    = & \eps \left( \nabla \gamma(x) + \eps \nabla \gamma_1(x) 
      \right) \left( F(x)+\eps F_y(x) \gamma_1(x) \right) 
      + \mathcal{O}(\eps^3)\\
    = & \eps \nabla \gamma(x) F(x) +\eps^2 \left( \nabla \gamma(x) 
      F_y(x) \gamma_1(x) + \nabla \gamma_1(x) F(x) \right) 
      + \mathcal{O}(\eps^3),
  \end{aligned}
\end{equation}
where $F_y(x) = [F_y^{i,j}(x)]_{i=1,j=1}^{n_x,n_y} = 
  \left[ \frac{\partial f^i}{\partial y^j}(x,\gamma(x)) 
  \right]_{i=1,j=1}^{n_x,n_y}$. 
By \cref{assump:dissipative}, $G_y(x)$ is invertible. 
By comparing \cref{equ:expansion_left} and 
\cref{equ:expansion_right}, one gets the analytic expressions of 
$\gamma_0(x)$, $\gamma_1(x)$ and $\gamma_2(x)$:
\begin{subequations}
\label{equ:expression_pre}
\begin{align}
\label{equ:expression_pre_0}
  &\gamma_0(x) = \gamma(x),\\
\label{equ:expression_pre_1}
  &\gamma_1(x) = G_y(x)^{-1} \nabla \gamma(x) F(x),\\
\label{equ:expression_pre_2}
  &\gamma_2(x) = G_y(x)^{-1} \Bigg( \nabla \gamma(x) F_y(x) 
  \gamma_1(x) \\
  & \nonumber \qquad\qquad\qquad + \nabla \gamma_1(x) F(x) - 
  \frac{1}{2} \Bigg[ \sum_{j,k=1}^{n_y} G_{yy}^{i,j,k} \gamma_1^j(x) 
  \gamma_1^k(x) \Bigg]_{i=1}^{n_y} \Bigg).
\end{align}
\end{subequations}
\begin{remark}
  \label{remark:expression}
  By taking gradient of both sides of 
  \cref{equ:assump_dissipative_1}, one obtains that 
  \begin{equation}
    \label{equ:nabla_gamma}
    G_x(x)+G_y(x)\nabla\gamma(x)=0,
  \end{equation}
  where $G_x(x) = [G_x^{i,j}(x)]_{i=1,j=1}^{n_y,n_x} = 
  \left[ \frac{\partial g^i}{\partial x^j} (x,\gamma(x)) \right] 
  _{i=1,j=1}^{n_y,n_x}$. 
  Therefore, one can obtain that $$\nabla\gamma(x) = 
  -G_y(x)^{-1}G_x(x).$$ 
  This implies that $\gamma_1(x)$, $\nabla\gamma_1(x)$ and 
  $\gamma_2(x)$ can be rewritten as algebraic expressions of function 
  values and derivatives of $f(x,y)$ and $g(x,y)$ at $(x,\gamma(x))$. 
  Therefore, one can obtain that $\gamma$, $\gamma_1$ and $\gamma_2$ 
  are sufficiently smooth. In addition, by 
  \cref{assump:bound,assump:dissipative}, and the expressions 
  in \cref{equ:expression_pre}, one can deduce that 
  $\nabla \gamma \in W^{K,\infty}$, $\gamma_1 \in 
  W^{K,\infty}$ and $\gamma_2 \in W^{K - 1, \infty}$. (See Lemmas in 
  \cref{app:thm:lemma}).
\end{remark}

\subsubsection{Properties}
As an approximation to $\Gamma(x,\eps)$, the function 
$\tilde{\Gamma}_2(x,\eps) := \gamma(x) + \eps \gamma_1(x) + \eps^2 
\gamma_2(x)$ should share similar properties with $\Gamma(x,\eps)$. 
Now we make the above statement rigorous. In this part, we let 
$\ms_\eps^{(2)} := \{ (x,y): y = \tilde{\Gamma}_2(x,\eps),\ x\in\rx\}$
and $z(t) = y(t) - \tilde{\Gamma}_2(x(t), \eps)$.


The attractive property that is parallel to \cref{prop:attractor} 
can 
be formulated as in \cref{thm:attractor}, which says that the system 
goes quickly from arbitrary initial values into a small vicinity of 
the approximate invariant manifold $\ms_\eps^{(2)}$.
\begin{theorem}
\label{thm:attractor}
There exists a constant $C>0$ independent of $\eps$ such that 
$$ |z(t)|^2 \le C \eps^6
+ e^{-\frac{\beta}{\eps} t} |z(0)|^2, $$ for sufficiently small 
$\eps$. In particular, as long as $t$ is sufficiently large, then 
$|z(t)| = \mathcal{O}(\eps^3)$.
\end{theorem}

\begin{proof}
By Taylor's expansion of $f(x,y)$ and $g(x,y)$ at $(x,\gamma(x))$, 
one gets that 
\begin{equation}
  \label{equ:thm_attractor_1}
  \begin{aligned}
  &\frac{1}{\eps} g(x,\tilde{\Gamma}_2(x,\eps)) 
  - \nabla_x \tilde{\Gamma}_2(x,\eps) 
    f(x,\tilde{\Gamma}_2(x,\eps)) \\
  =\ & G_y(x)\left( \gamma_1(x) + \eps \gamma_2(x) \right)
    + \frac{1}{2} \eps \left[ \sum_{j,k=1}^{n_y} G_{yy}^{i,j,k} 
    \gamma_1^j(x) \gamma_1^k(x) \right]_{i=1}^{n_y} \\
  &\quad\quad - \left( \nabla \gamma(x) + \eps \nabla \gamma_1(x)   
    \right) \left(F(x) + \eps F_y(x) \gamma_1(x) \right) + 
    \mathcal{O}(\eps^2)\\
  =\ & \left( G_y(x) \gamma_1(x) - \nabla \gamma(x) F(x) \right) \\
  &\quad\quad + \eps \Bigg( G_y(x) \gamma_2(x) + \frac{1}{2}
    \Bigg[ \sum_{j,k=1}^{n_y} G_{yy}^{i,j,k} \gamma_1^j(x) 
    \gamma_1^k(x) \Bigg]_{i=1}^{n_y} \\
    & \qquad\qquad\qquad\qquad\qquad - \nabla \gamma(x) F_y(x) 
    \gamma_1(x) - \nabla \gamma_1(x) F(x) \Bigg) + 
    \mathcal{O}(\eps^2) \\
  =\ &\mathcal{O}(\eps^2),
\end{aligned}
\end{equation}
where $\mathcal{O}(\eps^2)$ can be controlled 
uniformly thanks to \cref{assump:bound}. 
By \cref{assump:dissipative}, \cref{remark:expression} and 
\cref{equ:thm_attractor_1}, there exists
a constant $C_1>0$ such that
\begin{displaymath}
\begin{aligned}
  &\frac{1}{2} \od{|z|^2}{t} = \ip{z}{\od{z}{t}} 
    = \ip{y - \tilde{\Gamma}_2(x,\eps)}
    {\frac{1}{\eps}  g(x,y) - \nabla_x 
    \tilde{\Gamma}_2(x,\eps) f(x,y)} \\
  =& \ip{y - \tilde{\Gamma}_2(x,\eps)}{\frac{1}{\eps} g(x,y) - 
  	\frac{1}{\eps} g(x,\tilde{\Gamma}_2(x,\eps))} \\
  &\quad\quad + \ip{y - \tilde{\Gamma}_2(x,\eps)}
  {\frac{1}{\eps} g(x,\tilde{\Gamma}_2(x,\eps)) 
    - \nabla_x \tilde{\Gamma}_2(x,\eps) 
    f(x,\tilde{\Gamma}_2(x,\eps))} \\
  &\quad\quad + \ip{y - \tilde{\Gamma}_2(x,\eps)}
    { \nabla_x \tilde{\Gamma}_2(x,\eps) 
    (f(x,\tilde{\Gamma}_2(x,\eps)) - f(x,y))} \\
  \le& \left( -\frac{\beta}{\eps} + C_1 \right) |z|^2
    + C_1 \eps^2 |z|
    \le -\frac{\beta}{2\eps} |z|^2 + \frac{C_1^2}{\beta} \eps^5,
\end{aligned}
\end{displaymath}
where $0 < \eps \le \min\left\{ \frac{\beta}{4 C_1}, \eps_0 \right\}$.
Here the last inequality is due to the Cauchy–Schwarz inequality. By 
Gronwall's inequality, one gets that
\begin{displaymath}
|z(t)|^2 \le \frac{ 2 C_1^2 }{ \beta^2 } \eps^6 
(1 - e^{-\frac{\beta}{\eps} t}) + e^{-\frac{\beta}{\eps} t}|z(0)|^2,
\end{displaymath} 
which completes the proof.
\end{proof}

Now we consider the decouplable property. Substituting 
$\tilde{\Gamma}_2(x,\eps)$ for 
$\Gamma(x,\eps)$ in \cref{equ:invariant}, one obtains the following 
equation: 
\begin{equation}
  \label{equ:invariant_2}
  \od{X_2}{t} = f(X_2, \tilde{\Gamma}_2(X_2,\eps)).
\end{equation}
One may surmise that, if the initial values 
are sufficiently close to $\ms_\eps^{(2)}$, then one may 
use \cref{equ:invariant_2} to approximate the trajectory of $x(t)$ 
in some sense. To be rigorous, we have \cref{thm:invariant}.

\begin{theorem}
\label{thm:invariant}
There exists a constant $C > 0$ independent of $\eps$ such that
\begin{displaymath}
|x(t)-X_2(t)|^2 \le e^{Ct} \left( |x(0)-X_2(0)|^2 + \eps^6 
    + \frac{\eps}{\beta} |z(0)|^2 \right),
\end{displaymath}
for sufficiently small $\eps$.
In particular, if $|x(0) - X_2(0)| = \mathcal{O}(\eps^3)$ and
$|z(0)| = \mathcal{O}(\eps^{\frac{5}{2}})$, then 
$|x(t) - X_2(t)| = \mathcal{O}(\eps^3)$ for $t \sim \mathcal{O}(1)$.
\end{theorem}

\begin{proof}
  
  
Since $\nabla_x f(x,y)$, $\nabla_y f(x,y)$ and $\nabla_x 
\tilde{\Gamma}_2(x,\eps)$ are all bounded, then there exists a 
constant $C_1 > 0$ such that
\begin{displaymath}
\begin{aligned}
  &\left| \od{(x-X_2)}{t}\right | 
    = |f(x,y) - f(X_2,\tilde{\Gamma}_2(X_2,\eps))|\\
  \le\ &|f(x,z + \tilde{\Gamma}_2(x,\eps)) 
            - f(x,\tilde{\Gamma}_2(X_2,\eps))| 
        + |f(x,\tilde{\Gamma}_2(X_2,\eps)) 
            - f(X_2,\tilde{\Gamma}_2(X_2,\eps))|\\
  \le\ & C_1 |z| + C_1 |x - X_2|.
\end{aligned}
\end{displaymath}
Then one obtains that 
\begin{displaymath}
  \begin{aligned}
    &\frac{1}{2} \od{|x-X_2|^2}{t} = \ip{x-X_2}{\od{(x-X_2)}{t}} \\
    \le & |x-X_2| \cdot \left| \od{(x-X_2)}{t} \right| 
    \le C_1 |z| \cdot |x-X_2| + C_1 |x-X_2|^2.
  \end{aligned}
\end{displaymath}
Therefore, by \cref{thm:attractor} and Cauchy-Schwarz 
inequality, there exists a constant $C > 0$ such that 
\begin{displaymath}
\od{|x-X_2|^2}{t} \le C (|x-X_2|^2 + \eps^6) 
    + e^{-\frac{\beta}{\eps}t} |z(0)|^2.
\end{displaymath}
By Gronwall's inequality,
\begin{displaymath}
|x(t)-X_2(t)|^2 \le e^{Ct} |x(0)-X_2(0)|^2 + \eps^6 (e^{Ct} - 1)
    + \frac{e^{Ct} - e^{-\frac{\beta}{\eps}t}}
    {\frac{\beta}{\eps} + C} |z(0)|^2 ,
\end{displaymath}
which completes the proof.
\end{proof}

\begin{remark}
\label{remark:couple}
Now we look back on what \cref{thm:attractor,thm:invariant} tell us. 
\cref{thm:attractor} says that, 
no matter what the initial values are, the trajectory of 
$(x(t),y(t))$ always tends to a state where 
$|y - \tilde{\Gamma}_2(x,\eps)| = \mathcal{O}(\eps^3)$ at an 
exponential rate. \cref{thm:invariant} says that, once the 
trajectory arrives at this state, one may use 
\cref{equ:invariant_2}, which is not stiff, to approximate the 
trajectory of $x(t)$ in a finite time horizon with an accuracy of 
$\mathcal{O}(\eps^3)$. 
In the literature, this phenomenon is referred 
as an initial layer or a boundary layer
\cite{pavliotis2008multiscale, Roberts1989Appropriate, 
Verhulst2005Methods, Stephen1995Initial, legoll2013micro}.
This remark plays a guiding role in designing our numerical 
algorithms later. 
\end{remark}

\begin{remark}
\label{remark:complex}
The properties of $\tilde{\Gamma}_2(x,\eps)$ are studied in this 
part. One may guess that $\tilde{\Gamma}_k(x,\eps) := \sum_{j=0}^{k} 
\eps^j \gamma_j(x)$ should have similar properties. However, it is 
expected that the expression of $\gamma_k(x)$ is very complex when 
$k$ is large. 
In addition, the evaluation of high-order derivatives is involved 
in the expressions of $\gamma_k(x)$.
It is not convenient to conduct theoretical analysis or 
to design numerical algorithms at this time.
\end{remark}

\subsection{An iterative method}
\label{sec:model_iteration}
In this part, an iterative method is presented to approximate the 
invariant manifold $\ms_\eps$. We put forward an iterative formula, 
which can be used to produce a series of successively refined 
approximations to $\Gamma(x,\eps)$. This method overcomes the 
difficulties in \cref{remark:complex} due to its concise 
form. 
It can be proved theoretically that this method is consistent 
with 
asymptotic approximation in some sense, and that the 
approximation 
accuracy reaches $\mathcal{O}(\eps^{k+1})$ after $k$ iteration 
steps.

\subsubsection{An iterative formula}
Inspired by \cref{equ:center}, we propose the following fixed-point 
iterative formula to approximate $\Gamma(x,\eps)$:
\begin{subequations}
  \label{equ:iteration}
  \begin{align}
    \label{equ:iteration_1}
      & \Gamma_0(x,\eps) = \gamma(x), \\
    \label{equ:iteration_2}
      & g(x,\Gamma_{k+1}(x,\eps)) = \eps \nabla_x \Gamma_k(x,\eps) 
      f(x,\Gamma_k(x,\eps)),\ k=0,1,2,\ldots.
    \end{align}
\end{subequations}
\cref{remark:monotone} implies the existence and uniqueness of 
the solution $\Gamma_{k+1}(x,\eps)$ of \cref{equ:iteration_2}, since
$g(x,\cdot)$ is Lipschitz continuous and strongly dissipative. 
Using the notations in \cref{remark:monotone}, one can write 
\cref{equ:iteration_2} as
\begin{displaymath}
  \Gamma_{k+1}(x,\eps) = \tilde{g}_x^{-1}\left( \eps \nabla_x 
  \Gamma_k(x,\eps) f(x,\Gamma_k(x,\eps)) \right).
\end{displaymath}

\subsubsection{Relationships with asymptotic approximation}
In this part, we study the relationships between the iterative 
formula \cref{equ:iteration} and asymptotic approximation. It can be 
proved that the iterative formula can produce approximations in 
\cref{equ:expression_pre} in the first two iteration steps. The proof 
of \cref{thm:iteration} can be found in \cref{app:pf_iter}.
\begin{theorem}
\label{thm:iteration}
We have the following conclusions:
\begin{subequations}
  \begin{flalign}
  \label{thm:iteration_1}
    |\Gamma_1(x,\eps) - \gamma(x) - \eps \gamma_1(x)| 
        &= \mathcal{O}(\eps^2),\\
  \label{thm:iteration_2}
    |\nabla_x \Gamma_1(x,\eps) - \nabla \gamma(x) 
        - \eps \nabla \gamma_1(x)| &= \mathcal{O}(\eps^2),\\
  \label{thm:iteration_3}
    |\Gamma_2(x,\eps) - \gamma(x) - \eps \gamma_1(x) 
        - \eps^2 \gamma_2(x)| &= \mathcal{O}(\eps^3),
	\end{flalign}
\end{subequations}
where the bounds can be uniformly controlled in $x \in \rx$. In other 
words, 
\begin{displaymath}
  \begin{aligned}
    \norm{\Gamma_1(\cdot,\eps) - \gamma - \eps \gamma_1}_{1, \infty} 
    = \mathcal{O}(\eps^2), \\
    \norm{\Gamma_2(\cdot,\eps) - \gamma - \eps \gamma_1 - \eps^2 
    \gamma_2}_{0, \infty} = \mathcal{O}(\eps^3). \\
  \end{aligned}
\end{displaymath}
\end{theorem}

\subsubsection{High-order approximation and convergence}
We prove in \cref{thm:iteration} that the iterative formula 
\cref{equ:iteration} matches the formal expansions in the first two 
iteration steps. As for high-order approximation, it should be 
tedious to prove similar results due to the complex expressions of 
$\gamma_k(x)$ when $k$ is large. However, we can get rid of the 
specific expressions of $\gamma_k(x)$ and prove the attractive 
property and the decouplable property that are parallel to
\cref{thm:attractor,thm:invariant}. Proofs of the 
following two theorems can be found in \cref{app:thm_iteration}. The 
interpretations for these properties are 
similar to \cref{remark:couple}. In this part, we let $z_k(t) = 
y(t) - \Gamma_k(x(t), \eps)$ and let $X_k$ be defined as the solution 
of 
\begin{equation} \label{equ:decouple}
    \od{X_k}{t}=f(X_k,\Gamma_k(X_k,\eps)).
\end{equation}
\begin{theorem}
	\label{thm:iteration_attractor}
    For any $k = 0, 1, \ldots, K$ and for any $A \in (0,2\beta)$, 
    there exists a constant $C_k>0$ such that
	\begin{displaymath}
	|z_k(t)|^2 \le C_k (\eps^{2k+2} + e^{-\frac{A}{\eps}t} 
    |z_k(0)|^2),
	\end{displaymath}
	where $C_k$ is dependent on $k$ and independent of $\eps$. 
	In particular, as long as $t$ is sufficiently large, then 
    $|z_k(t)| = \mathcal{O}(\eps^{k+1})$.
\end{theorem}

\begin{theorem}
	\label{thm:iteration_invariant}
  For any $k = 0, 1, \ldots, K$, there exists a constant $C_k > 0$ 
  dependent on $k$ and independent of $\eps$ such that
  \begin{displaymath}
      |x(t)-X_k(t)|^2 \le e^{C_k t} \left( |x(0)-X_k(0)|^2 
      + \eps^{2k + 2} + \frac{C_k \eps}{\beta} |z_k(0)|^2 \right).
  \end{displaymath}
  In particular, if $|x(0) - X_k(0)| = \mathcal{O}(\eps^{k+1})$ and
  $|z_k(0)| = \mathcal{O}(\eps^{k + \frac{1}{2}})$, then 
  $|x(t) - X_k(t)| = \mathcal{O}(\eps^{k+1})$ for $t \sim 
  \mathcal{O}(1)$.
\end{theorem}

Now we consider using the iterative formula \cref{equ:iteration} on 
the linearized equation \cref{equ:linear}. One may notice that 
\cref{equ:linear} does not satisfy \cref{assump:bound} 
since $f(x,y)$ in this case is not bounded. However, we can still 
prove similar results that the iterative formula 
\cref{equ:iteration} improves the approximation accuracy. Actually, 
we can also prove the convergence of \cref{equ:iteration} in this 
case.
\begin{example}
	\label{example:linear_2}
	We continue considering the linear equation \cref{equ:linear} in 
	\cref{example:linear_1}. By \cref{thm:linear}, there 
	exist uniform bounded solutions $(C^*,d^*)$ to 
	\cref{equ:center_linear} for sufficiently small $\eps$. Suppose 
	that $\Gamma_k(x,\eps)=C_kx+d_k$, where 
	$C_k=C_k(\eps)\in\mathbb{R}^{n_y\times n_x}$ and 
	$d_k=d_k(\eps)\in\ry$. By \cref{equ:iteration}, the iterative 
	formula for $C_k$ and $d_k$ can be given as
	\begin{displaymath}
		\begin{aligned}
		&C_0=-A_{22}^{-1}A_{21},\ d_0=-A_{22}^{-1}b_2,\\
		&C_{k+1}=-A_{22}^{-1}A_{21}+\eps A_{22}^{-1}C_kA_{11}+\eps A_{22}^{-1}C_kA_{12}C_k,\\
		&d_{k+1}=-A_{22}^{-1}b_2+\eps A_{22}^{-1}C_kb_1+\eps A_{22}^{-1}C_kA_{12}d_k.
		\end{aligned}
	\end{displaymath}
	Therefore, one can deduce that 
	\begin{subequations}
	\label{equ:iteration_linear}
	\begin{align}
	\label{equ:iteration_linear_1}&C_{k+1}-C^*=\eps A_{22}^{-1}(C_k-C^*)A_{11}+\eps A_{22}^{-1}C_kA_{12}(C_k-C^*)+\eps A_{22}^{-1}(C_k-C^*)A_{12}C^*,\\
	\label{equ:iteration_linear_2}&d_{k+1}-d^*=\eps A_{22}^{-1}(C_k-C^*)b_1+\eps A_{22}^{-1}C_kA_{12}(d_k-d^*)+\eps A_{22}^{-1}(C_k-C^*)A_{12}d^*.
	\end{align}
	\end{subequations}
	First, we consider \cref{equ:iteration_linear_1}. By the uniform 
	boundedness of $C^*$, there exists a constant $M_1>0$ independent 
	of $\eps$ and $k$ such that
	\begin{equation}
	\label{equ:example_linear_1}
	\norm{C_{k+1}-C^*}\le \eps M_1\norm{C_k-C^*}+\eps M_1\norm{C_k-C^*}^2.
	\end{equation}
	If we take $\eps$ sufficiently small, for example, such that $\eps 
	M_1(1+\norm{C_0-C^*})\le \frac{1}{2}$, then one can obtain by 
	\cref{equ:example_linear_1} that 
	\begin{displaymath}
	\norm{C_{k+1}-C^*}\le\frac{1}{2}\norm{C_k-C^*},\ \forall k=0,1,2,\ldots.
	\end{displaymath}
	Therefore, the sequence $\{C_k\}$ converges to $C^*$ and is 
	uniformly bounded in $k$ and sufficiently small $\eps$. By 
	\cref{equ:iteration_linear_1}, there exists a constant $M_2>0$ 
	independent of $\eps$ and $k$ such that 
	\begin{displaymath}
	\norm{C_{k+1}-C^*}\le\eps M_2\norm{C_k-C^*}.
	\end{displaymath}
	By \cref{thm:linear}, $\norm{C_0-C^*}=\mathcal{O}(\eps)$. 
	Therefore, $\norm{C_k-C^*}=\mathcal{O}(\eps^{k+1})$ for fixed $k$. 
	
	By \cref{equ:iteration_linear_2}, there exists a constant $M_3>0$ 
	independent of $\eps$ and $k$ such that 
	\begin{equation}
	\label{equ:example_linear_2}
	|d_{k+1}-d^*|\le\eps M_3\norm{C_{k}-C^*}+\eps M_3 |d_k-d^*|,
	\end{equation}
	which indicates that $|d_k-d^*|$ is uniformly bounded when $\eps$ is sufficiently small. Thus, $\varlimsup\limits_{k\rightarrow\infty}|d_k-d^*|<+\infty$.
	By taking upper limit of \cref{equ:example_linear_2}, one gets that 
	\begin{displaymath}
	(1-\eps M_3)\varlimsup\limits_{k\rightarrow\infty}|d_k-d^*|\le 0.
	\end{displaymath}
	Therefore, when $\eps$ is sufficiently small, the sequence 
	$\{d_k\}$ converges to $d^*$. By \cref{thm:linear}, 
	$|d_0-d^*|=\mathcal{O}(\eps)$. Therefore, by 
	\cref{equ:example_linear_2}, one obtains that 
	$|d_k-d^*|=\mathcal{O}(\eps^{k+1})$ for fixed $k$.
	
	In conclusion, we have already proven that $C_k\rightarrow C^*$ and $d_k\rightarrow d^*$ as $k\rightarrow\infty$. Furthermore, for fixed $k$, we have that $\norm{C_k-C^*}=\mathcal{O}(\eps^{k+1})$ and $|d_k-d^*|=\mathcal{O}(\eps^{k+1})$.
\end{example}

%% file: article_algorithm.tex
\section{Numerical Scheme}
\label{sec:alg}


In this section, we develop numerical algorithms to implement the 
models in \Cref{sec:model}.

\subsection{Basic framework and notations}
\label{sec:alg_frame}
Basically, our numerical scheme falls into the framework of HMM, 
which contains a microscopic solver to compute the steady state, and 
a macroscopic solver to evaluate the slow variables. 
As seen in \cref{remark:couple}, the numerical simulations can 
be divided into two stages:
\begin{itemize}
  \item the first stage: solving the coupled system \cref{equation} 
  in the initial layer by a coupled solver until $|z_k(t)| = |y(t) - 
  \Gamma_k(x(t), \eps)|$ is sufficiently small;
  \item the second stage: using the HMM-type algorithms developed 
  later to solve the decoupled system \cref{equ:decouple}.
\end{itemize}
Briefly, the HMM-type algorithms contain two parts: 
\begin{itemize}
  \item approximating the invariant manifold, i.e., calculating 
  $\Gamma_k(x, \eps)$ by a microscopic solver;
  \item solving the decoupled system of the slow variables by a 
  macroscopic solver.
\end{itemize}


In our numerical scheme, we need to calculate $\Gamma_k(x,\eps)$, as 
mentioned before. Let $\hg_k(x,\eps) \approx \Gamma_{k}(x,\eps)$ be 
the numerical approximation to $\Gamma_k(x,\eps)$. Hereinafter, we 
denote the algorithms, where $\Gamma_k(x,\eps)$ is needed, by HMM$k$.
Given a grid $\{ t_n \}$, let $(x_n,y_n)$ be the numerical 
approximation to $(x(t_n), y(t_n))$.

In the first stage, we solve the coupled system \cref{equation} 
numerically on the grid $\{ t_n = n \Delta t_c \}_{n = 0}^{N_c}$ with 
$N_c \Delta t_c = T_c$, 
where $\Delta t_c$ is the time step size of the coupled solver and 
$N_c$ can be determined by the criterion in \cref{remark:terminal}. 
One may choose an explicit one-step scheme 
as the coupled solver, which can be written as
\footnote{
  Notice that an explicit one-step scheme for the ODE:
  $ \od{z}{t} = h(z) $
  can be written in the form of 
  $ z_{n+1} = z_n + \Delta t \phi(z_n, h, \Delta t) $.
}
\begin{displaymath}
  \left( x_{n+1}, y_{n+1} \right) = \left( x_n, y_n \right) + \Delta 
  t_c \phi_c\left( \left( x_n, y_n \right), \left( f, g / \eps 
  \right), \Delta t_c \right), \quad n = 0, 1, \ldots, N_c - 1.
\end{displaymath}

In the second stage, we solve the decoupled system 
\cref{equ:decouple} numerically with initial value $ x_{N_c} $ on the 
grid$ \{ t_n = T_c + (n - N_c) \Delta t \}_{n = N_c}^N $ with $(N - 
N_c) \Delta t = T - T_c$, where $\Delta t$ is the time step size of 
the macroscopic solver. Similarly, we may choose an explicit one-step 
scheme as the macroscopic solver:
\begin{displaymath}
  x_{n+1} = x_n + \Delta t \phi_d(x_n, f(\cdot,\hg_k(\cdot,\eps)), 
  \Delta t), \quad n = N_c, N_c + 1, \ldots, N - 1.
\end{displaymath}

In our numerical scheme, we may need to solve the equation 
\begin{equation}   \label{equ:microequ}
  g(\tx,\ty) = \eps h
\end{equation}
with respect to $\ty$, where $h$ is a given quantity which may 
depend on $\tx$. \cref{remark:monotone} ensures the existence 
and uniqueness of the solution of \cref{equ:microequ}. Let us denote 
the solution by $ \tilde{g}_{\tx}^{-1}(\eps h) $. Consider the 
following ODE with respect 
to $\ty$
\begin{equation}  \label{equ:microode}
  \od{\ty}{t} = \frac{1}{\eps} g(\tx,\ty) - h,
\end{equation}
whose stationary point is exactly $\tilde{g}_{\tx}^{-1}(\eps h)$.
The microscopic solver solves \cref{equ:microode} numerically:
\begin{displaymath}
  \begin{aligned}
    & \ty_{m+1} = \ty_m + \delta t \phi_m(\ty_m, g(\tilde{x},\cdot) / 
    \eps - h, \delta t), \quad m = 0, 1, \ldots, M - 1, \\
    & \ty_0 \text{ suitably chosen},
  \end{aligned}
\end{displaymath}
where $\ty_m$ is short for $\ty_m(\tx, h)$, $\delta t$ is the 
time step size and $M$ is number of steps in the microscopic solver. 
One may estimate $\tilde{g}_{\tx}^{-1}(\eps h)$ by $\ty_M(\tx, 
h)$.
We will discuss the selection of the initial value $\ty_0$ in 
\cref{remark:initial}.

\subsection{Initial layer}
\label{sec:alg_initial}
In the initial layer $[0,T_c]$, i.e., the first stage of 
simulation, the system tends to the invariant manifold quickly. 
According to \cref{thm:iteration_invariant}, the terminal time 
$T_c$ of the coupled solver should be taken such that 
\begin{displaymath}
  |y(T_c) - \Gamma_k(x(T_c), \eps)| = \mathcal{O}(\eps^{k + 
  \frac{1}{2}}).
\end{displaymath}
By \cref{thm:iteration_attractor}, $T_c$ should be of order 
$\mathcal{O}(\eps \log \frac{1}{\eps}) $ for fixed initial value and 
$k$. This indicates that the coupled solver is only needed for a 
short time.

\begin{remark}
  \label{remark:terminal}
  It is a subtle problem to determine $T_c$ in numerical 
  simulations. 
  Here we present a possible empirical criterion inspired by 
  \cref{thm:iteration_attractor}:
  if
  \begin{displaymath}
    \begin{aligned}
      & n \equiv 0\quad (\mathrm{mod}\ n_p) \\
      & |y_n - \hg_k(x_n, \eps)| \ge \mu |y_{n - n_p} - \hg_k(x_{n - 
        n_p}, 
      \eps)|,
    \end{aligned}
  \end{displaymath}
  for some positive integer $n$, where $n_p \in 
  \mathbb{N} \setminus \{ 0 \}$ are given beforehand, then terminate 
  the coupled solver and let $N_c$ be $n$. We let $\mu = 
  \exp\left( -\frac{\hat{\beta}}{2\eps} n_p \Delta t_c \right)$, 
  where $- \hat{\beta} < 0$ is an estimation of the upper bound 
  of the eigenvalues of $\pd{g}{y}(x,y)$. We 
  calculate $\hg_k$ per $n_p$ steps in order to reduce 
  computational cost. In our numerical experiments, we set $n_p = 10$.
\end{remark}

\subsection{Approximation to the invariant manifold}
In this part, we present several numerical approaches to 
approximating $\Gamma_k(x,\eps)$.

\subsubsection{A naive approach}
\label{sec:alg1}
\cref{remark:expression} tells us that $\gamma_1(x)$ and 
$\gamma_2(x)$ can be rewritten as expressions 
of function values and derivatives of $f(x,y)$ and $g(x,y)$ at 
$(x,\gamma(x))$. In addition, $\gamma(x)$ can be approximated by 
\begin{equation}  \label{equ:gamma0}
  \gamma(x) \approx \ty_M(x, 0).
\end{equation}
Therefore, $\Gamma_k(x, \eps)$ can be 
approximated according to the expressions \cref{equ:expression_pre} 
when $k = 0, 1, 2$. However, these expressions are too complicated. 
In this part, we would like to design an algorithm that is easy to 
implement. 
The algorithm is based on the analytical expressions in 
\cref{equ:expression_pre} and it requires to evaluate the Jacobian 
matrix of $g(x,y)$.

First, we notice by \cref{equ:expression_pre_1} and 
\cref{equ:nabla_gamma} that
\begin{equation} \label{equ:gamma1}
\gamma_1(x) = -G_y(x)^{-2} G_x(x) F(x).
\end{equation}
By \cref{equ:expression_pre}, one can obtain that 
\begin{equation}
\label{equ:alg_expansion}
\begin{aligned}
\Gamma_2(x,\eps) =\ & \gamma_0(x) + \eps \gamma_1(x) + \eps^2 
\gamma_2(x) + \mathcal{O}(\eps^3)\\
=\ & \gamma(x) + \eps G_y(x)^{-1} \left( \nabla \gamma(x) + \eps 
\nabla \gamma_1(x) \right) f(x,\gamma(x) + \eps \gamma_1(x))\\
& \quad \quad -\frac{1}{2} \eps^2 G_y(x)^{-1} \left[ 
\sum_{j,k=1}^{n_y} G_{yy}^{i,j,k} \gamma_1^j(x) \gamma_1^k(x) 
\right]_{i=1}^{n_y} + \mathcal{O}(\eps^3).
\end{aligned}
\end{equation}
Now we consider how the terms in \cref{equ:alg_expansion} are 
evaluated. One can notice that 
\begin{displaymath}
\begin{aligned}
  & \left( \nabla \gamma(x) + \eps \nabla \gamma_1(x) \right) 
  f(x,\gamma(x) + \eps \gamma_1(x)) \\
  & \qquad \qquad = \lim \limits_{\tau \rightarrow 0} \frac{\gamma(x 
  + F_1(x) \tau) + \eps \gamma_1(x + F_1(x) \tau) - \gamma(x) - \eps 
  \gamma_1(x)}{\tau},
\end{aligned}
\end{displaymath}
where $F_1(x):=f(x, \gamma(x) + \eps \gamma_1(x))$.
Therefore, numerical derivatives can be used to approximate this 
term, that is,
\begin{equation}
\label{equ:alg_expansion_1}
\left( \nabla \gamma(x) + \eps \nabla \gamma_1(x) \right) 
f(x,\gamma(x) + \eps \gamma_1(x)) \approx \frac{\Delta_{+} 
y}{\tau},
\end{equation}
where $\Delta_{+} y = \gamma(x+F_1(x)\tau) + \eps 
\gamma_1(x+F_1(x)\tau) - \gamma(x) - \eps\gamma_1(x)$ with $\tau$ 
suitably chosen. Besides, Taylor's expansion of $g(x,y)$ at 
$(x,\gamma(x))$ yields that 
\begin{equation}
\label{equ:alg_expansion_2}
g(x,\gamma(x)+\eps\gamma_1(x))=\eps G_y(x)\gamma_1(x)+\frac{1}{2}\eps^2\Big[ \sum_{j,k=1}^{n_y} G_{yy}^{i,j,k} \gamma_1^j(x) \gamma_1^k(x) \Big]_{i=1}^{n_y}+\mathcal{O}(\eps^3).
\end{equation}
By \cref{equ:alg_expansion}, \cref{equ:alg_expansion_1} and 
\cref{equ:alg_expansion_2}, one can obtain that
\begin{displaymath}
\Gamma_2(x,\eps) \approx \gamma(x) + \eps \gamma_1(x) + G_y(x)^{-1} 
\left( \eps \frac{\Delta_{+} y}{\tau} - g(x,\gamma(x) + \eps 
\gamma_1(x)) \right),
\end{displaymath}
where $\gamma(x)$ and $\gamma_1(x)$ can be approximated by 
\cref{equ:gamma0} and \cref{equ:gamma1}, respectively. 

\subsubsection{Based on the iterative formula}
\label{sec:alg2}
Now we would like to utilize the iterative formula 
\cref{equ:iteration} to design an high-order HMM-type algorithm.


In order to obtain $\Gamma_k (x, \eps)$, we need to 
approximate 
$\nabla_x \Gamma_{k - 1}(x, \eps) f(x, \Gamma_{k - 1}(x, \eps))$. 
Actually, this 
term is the directional derivative of $\Gamma_{k - 1}(x, \eps)$ along 
$ f(x, \Gamma_{k - 1}(x, \eps)) $. Similar to what we did previously 
in 
\Cref{sec:alg1}, we notice that 
\begin{displaymath}
  \begin{aligned}
    &\nabla_x \Gamma_{k - 1}(x, \eps) f(x, \Gamma_{k - 1}(x, \eps)) 
    \\ & \quad \quad = \lim 
    \limits_{ 
      \tau \rightarrow 0} \frac{ \Gamma_{k - 1}(x + f(x, \Gamma_{k - 
      1}(x, 
      \eps)) \tau, 
      \eps) - \Gamma_{k - 1}(x,\eps) }{\tau}.
  \end{aligned}
\end{displaymath}
Therefore, one may use numerical derivative to approximate this term, 
that is,
\begin{displaymath}
\nabla_x \Gamma_{k - 1}(x, \eps) f(x, \Gamma_{k - 1}(x, \eps)) 
\approx \frac{ 
\Delta_{+} y }{\tau},
\end{displaymath}
where $\Delta_{+} y  = \Gamma_{k - 1} (x + f(x, \Gamma_{k - 
1}(x, \eps)) \tau, \eps) - \Gamma_{k - 1}(x, \eps)$ with $\tau$ 
suitably chosen. 
Then we 
need to solve an equation in the form of \cref{equ:microequ}, whose 
solution can be approximated by the microscopic solver.

\begin{remark}
  Using the iterative formula, we can obtain high-order 
  $\Gamma_k(x,\eps)$ by recursion. Actually, one can use the 
  algorithm introduced in \Cref{sec:alg1} to obtain 
  $\Gamma_k(x,\eps)$ where $k \ge 3$ by recursion in a similar way.
\end{remark}

\subsubsection{Summary and several remarks}
We summarize the algorithms introduced in \Cref{sec:alg1,sec:alg2} in 
\cref{alg1,alg2}, respectively.


\begin{algorithm}[ht]
  \caption{Approximating $\Gamma_k(x,\eps)$ using the methods 
  in \Cref{sec:alg1}.}
  \label{alg1}
  \begin{algorithmic}
    \Function{HMMtype1}{$x$, $\eps$, $k$}
      \If{$k$ equals to $0$}
        \State $\hat{\Gamma}_0(x,\eps) \gets \ty_M(x, 0)$;
        \State \Return{$\hat{\Gamma}_0(x,\eps)$};
      \ElsIf{$k$ equals to $1$}
        \State $\hat{\gamma}(x) \gets \ty_M(x, 0)$;
        \State $\hat{\gamma}_1(x)$ according to 
        \cref{equ:gamma1};
        \State $\hat{\Gamma}_1(x,\eps) \gets \hat{\gamma}(x) 
        + \eps \hat{\gamma}_1(x)$;
        \State \Return{$\hat{\Gamma}_1(x,\eps)$};
      \ElsIf{$k$ equals to $2$}
        \State $\hat{\gamma}(x) \gets \ty_M(x, 0)$ and $\hat{G}_y(x) 
        \gets \pd{g}{y}(x,\hat{\gamma}(x))$;
        \State $\hat{\gamma}_1(x)$ according to 
        \cref{equ:gamma1};
        \State $\hat{\Gamma}_1(x,\eps) \gets \hat{\gamma}(x) + \eps 
        \hat{\gamma}_1(x)$;
        \State $\hat{F}_1(x) \gets f(x, \hat{\Gamma}_1(x,\eps))$;
        \State $\hat{\Gamma}_1(x+\hat{F}_1(x)\tau, \eps) \gets 
        $\Call{HMMtype1}{$x+\hat{F}_1(x)\tau$, $\eps$, $1$};
        \State $\hat{\Gamma}_2(x,\eps) \gets \hat{\Gamma}_1(x,\eps) + 
        \hat{G}_y(x)^{-1} \left( \eps \frac{\hat{\Gamma}_1(x + 
        \hat{F}_1(x) \tau, \eps) - \hat{\Gamma}_1(x,\eps)}{\tau} - 
        g(x,\hat{\Gamma}_1(x,\eps)) \right)$;
        \State \Return{$\hat{\Gamma}_2(x,\eps)$};
      \Else
        \State $\hat{\Gamma}_{k-1}(x,\eps) \gets 
        $\Call{HMMtype1}{$x$, $\eps$, $k-1$};
        \State $\hat{F}_{k-1}(x) \gets 
        f(x,\hat{\Gamma}_{k-1}(x,\eps))$;
        \State $\hat{\Gamma}_{k-1}(x+\hat{F}_{k-1}(x)\tau,\eps) 
        \gets$\Call{HMMtype1}{$x+\hat{F}_{k-1}(x)\tau$, $\eps$, 
        $k-1$};
        \State $\hat{\Gamma}_k(x,\eps) \gets \ty_M \left(x, 
        \frac{\hg_{k - 1}(x + \hat{F}_{k-1}(x) \tau, \eps) - \hg_{k - 
        1}(x, \eps)}{\tau} \right)$;
        \State \Return{$\hat{\Gamma}_k(x,\eps)$};
      \EndIf
    \EndFunction
  \end{algorithmic}
\end{algorithm}

\begin{algorithm}[ht]
  \caption{Approximating $\Gamma_k(x,\eps)$ using the methods 
    in \Cref{sec:alg2}.}
  \label{alg2}
  \begin{algorithmic}
    \Function{HMMtype2}{$x$, $\eps$, $k$}
      \If{$k$ equals to $0$}
        \State $\hg_0(x,\eps) \gets \ty_M(x,0)$;
        \State \Return{$\hg_0(x,\eps)$};
      \Else
        \State $\hg_{k-1}(x,\eps) \gets $\Call{HMMtype2}{$x$, $\eps$, 
        $k-1$};
        \State $\hat{F}_{k-1}(x) \gets f(x,\hg_{k - 1}(x,\eps))$;
        \State $\hg_{k-1}(x+\hat{F}_{k-1}(x)\tau,\eps) 
        \gets$\Call{HMMtype2}{$x+\hat{F}_{k-1}(x)\tau$, $\eps$, 
        $k-1$};
        \State $\hg_{k}(x,\eps) \gets \ty_M \left(x, \frac{\hg_{k - 
        1}(x + \hat{F}_{k-1}(x) \tau, \eps) - \hg_{k - 1}(x, 
        \eps)}{\tau} \right)$;
        \State \Return{$\hg_k(x,\eps)$};
      \EndIf 
    \EndFunction
  \end{algorithmic}
\end{algorithm}

\begin{remark}
	These two algorithms are recursive algorithms. In order to evaluate 
	$\hat{\Gamma}_k(x,\eps)$ for one time, \cref{alg2} needs to call 
	the microscopic solver for $(2^{k+1} - 1)$ times, while \cref{alg1} 
	needs one time if $k \le 1$, and $(3 \times 2^{k-2} - 1)$ times 
	otherwise. 
  One can notice that \cref{alg1} calls the microscopic 
  solver for fewer times than \cref{alg2}, at the expense of 
  computing the Jacobian matrix of $g(x,y)$. The computation cost 
  increases exponentially as $k$ increases.
  In addition, a larger $k$ may lead to more accumulations of 
  numerical error (See \cref{remark:accum}). 
\end{remark}

\begin{remark}
  \label{remark:initial}
  Another advantage of high-order HMM is that it can give a better 
  estimation of the fast variable $y$ at the next time step, as the 
  initial value of the microscopic solver. In both algorithms, when 
  calculating $\hat{\Gamma}_k(x,\eps)$, we need to approximate the 
  directional derivative of $\Gamma_{k-1}(x,\eps)$ along 
  $f(x,\Gamma_{k-1}(x,\eps))$. For example, in \cref{alg2} at a 
  microscopic time step $t = t_n$, we need $\frac{\Delta_{+} 
  y}{\tau}$ 
  to approximate $\nabla_x \Gamma_{k-1}(x_n, \eps) f(x_n, \Gamma_{k - 
  1}(x_n, \eps))$. At the next time step $t = t_{n+1}$, one can use 
  $\hg_k(x_n, \eps) + \frac{\Delta_{+} y}{\tau} \Delta t$ as 
  $y_{n+1}$, which may reduce sampling error. Similar strategies can 
  be used in a Runge-Kutta type macroscopic solver.
\end{remark}

\begin{remark}
  \label{remark:central}
  In both algorithms, we use the forward difference formula to 
  approximate the directional derivatives. We may use higher-order 
  difference formulas for higher accuracy, for example, the central 
  difference formula. See \cref{remark:hd}.
\end{remark}

%% file: article_analysis.tex
\section{Numerical Analysis}
\label{sec:analysis}

Now we present some results of numerical analysis on our algorithms. 
We focus on \cref{alg2}. Some results here are also 
applicable to \cref{alg1}. 

As mentioned before, there are three main sources of errors of HMM, 
including modeling error, sampling error and truncation error of the 
macroscopic solver. We have proven that the modeling error can be 
reduced to $\mathcal{O}(\eps^{k+1})$. The truncation error of the 
macroscopic solver depends on what the macroscopic solver is. In 
classical HMM, sampling error mainly comes from the microscopic 
solver. Numerical analysis of sampling error in classical HMM can be 
found in \cite{Weinan2003Analysis}. In our algorithms, another source 
of sampling error is numerical derivatives. In addition, numerical 
error may accumulate in our recursive algorithms.

\subsection{Numerical derivatives}
\label{sec:analysis_nd}
Now we analyze the sampling error generated by the numerical 
derivatives directly. In this part, the round-off error and error 
from the microscopic solver is disregarded. In other words, we would 
like to analyze the error between $\Gamma_k(x,\eps)$ and 
$\hgd_k(x,\eps)$, where $\Gamma_k(x,\eps)$ and $\hgd_k(x,\eps)$ are 
defined as follows:
\begin{equation}\label{equ:iter_nd}
	\begin{aligned}
    &\Gamma_0(x,\eps) = \hgd_0(x,\eps)=\gamma(x).& \\
		&g(x,\Gamma_{k+1}(x,\eps)) = 
		\eps\nabla_x\Gamma_k(x,\eps)f(x,\Gamma_k(x,\eps)),&\
		 k\in\mathbb{N},\\
		&g(x,\hgd_{k+1}(x,\eps)) = 
		\eps\frac{\hgd_k(x+f(x,\hgd_k(x,\eps))\tau,\eps) - 
		\hgd_k(x,\eps)}{\tau},&\
		 k\in\mathbb{N}.\\
	\end{aligned}
\end{equation}
The last equality of \cref{equ:iter_nd} can also be written as 
\begin{displaymath}
  \hat{\Gamma}_{k+1}^\rd(x,\eps) = \tilde{g}_x^{-1} \left( 
  \eps\frac{\hgd_k(x+f(x,\hgd_k(x,\eps))\tau,\eps) - 
    \hgd_k(x,\eps)}{\tau} \right).
\end{displaymath}
We have the following result. The proof of \cref{thm:nd} can 
be found in \cref{app:thm_iteration}.
\begin{theorem}
  \label{thm:nd}
  For any $k = 0, 1, \ldots, \lfloor \frac{K}{2} \rfloor$,
  \begin{equation}
    \label{equ:nd}
    \norm{ \Gamma_k(\cdot,\eps) - \hgd_k(\cdot,\eps) 
    }_{K-2k,\infty} = \mathcal{O}(\eps\tau).
  \end{equation}
\end{theorem}

\begin{remark}
  \cref{thm:nd} is a little counter-intuitive. A direct 
  analysis may go as follows. Let $e_k^\rd =  \norm{\Gamma_k(\cdot, 
  \eps) - \hgd_k(\cdot,\eps)}_{0,\infty} $, then
  \begin{equation*}
    \begin{aligned}
      g(x, \hgd_{k+1}(x,\eps)) & = \eps \frac{\hgd_k(x + 
      f(x,\hgd_k(x,\eps)) \tau, \eps) - \hgd_k(x,\eps)}{\tau}\\
      & = \eps \frac{{\Gamma}_k(x + f(x,{\Gamma}_k(x, \eps))\tau, 
      \eps) - {\Gamma}_k(x,\eps)}{\tau} + \mathcal{O}(\frac{\eps
        e_k^\rd}{\tau})\\
      & = \eps \nabla_x \Gamma_k (x, \eps) f(x, \Gamma_k(x, \eps)) + 
      \mathcal{O}(\frac{\eps e_k^\rd}{\tau} + \eps \tau)\\
      & = g(x, \Gamma_{k+1}(x,\eps)) + \mathcal{O} (\frac{\eps 
      e_k^\rd} 
      {\tau} + \eps \tau).
    \end{aligned}
  \end{equation*}
  Therefore, $e_{k+1}^\rd = \mathcal{O}(\frac{\eps e_k^\rd}{\tau} + 
  \eps 
  \tau)$. In this way, one can calculate that 
  \begin{equation*}
    e_0^\rd = 0,\ e_1^\rd = \mathcal{O}(\eps\tau),\ 
    e_2^\rd = \mathcal{O}(\eps^2+\eps\tau),\ 
    e_3^\rd = \mathcal{O}(\frac{\eps^3}{\tau}+\eps^2+\eps\tau),\ 
    \ldots.
  \end{equation*}
  However, this is not a good estimate. For example, when $k=2$, the 
  modeling error turns out to be $\mathcal{O}(\eps^3)$, while the 
  bound for the sampling error is $\mathcal{O}(\eps^2+\eps\tau)$. 
  Actually, if we have some additional assumptions on the regularity 
  of $f$ and $g$, then we can improve the bound for this type of 
  error to $\mathcal{O}(\eps\tau)$. 
\end{remark}

\begin{remark}
  \label{remark:hd}
  In the spirit of \cref{remark:central}, we can use central 
  difference formula rather than forward difference formula. Notice 
  \cref{lemma:highOrderD_7} in \cref{lemma:highOrderD},  
  we can prove that for any $k = 0, 1, \ldots, \lfloor \frac{K}{3} 
  \rfloor$,
  \begin{equation*}
    \norm{ \Gamma_k(\cdot,\eps) - \hgd_k(\cdot,\eps) 
    }_{K-3k,\infty} = \mathcal{O}(\eps\tau^2),
  \end{equation*}
  where $\hgd_k$ is replaced by central difference formula. The proof 
  is exactly like that of \cref{thm:nd}.
\end{remark}

\subsection{Accumulation of numerical error}
\label{sec:ana_accum}
Due to the round-off error and error from the microscopic solver, one 
cannot obtain $\hat{\Gamma}^\rd_k(x,\eps)$ exactly as in 
\cref{equ:iter_nd}. This may lead to accumulation of numerical error. 
In this part, we aim to analyze the error between 
$\hat{\Gamma}_k^\rd$ and $\hat{\Gamma}_k^\rr$, where 
$\hat{\Gamma}_k^\rd$ is defined in \cref{equ:iter_nd}, and 
$\hat{\Gamma}_k^\rr$ satisfies that 
\begin{displaymath}
  \begin{aligned}
    & \norm{\hat{\Gamma}_0^\rr(\cdot,\eps) - \gamma}_{0,\infty} = 
    \eta_0, 
    \\
    &  \norm{\hat{\Gamma}_{k+1}^\rr(\cdot,\eps) - 
    \tilde{g}_{x}^{-1} \left( 
    \eps \frac{\hat{\Gamma}_k^\rr(\cdot + f(\cdot, 
    \hat{\Gamma}_k^\rr(\cdot,\eps)) \tau, \eps) - 
    \hat{\Gamma}_k^\rr(\cdot,\eps)}{\tau} \right)}_{0,\infty}  = 
    \eta_{k+1},\ k \in \mathbb{N}. 
  \end{aligned}
\end{displaymath}

\begin{theorem}\label{thm:accum}
  For any $k = 0, 1, \ldots, \lfloor \frac{K+1}{2} \rfloor$, 
  \begin{displaymath}
    \norm{\hat{\Gamma}_k^\rr(\cdot, \eps) - \hat{\Gamma}_k^\rd(\cdot, 
    \eps)}_{0, \infty} = \mathcal{O}\left( \sum_{j=0}^{k} 
    \frac{\eps^j}{\tau^j}\eta_{k-j} \right).
  \end{displaymath}
\end{theorem}
\begin{proof}
  Let $e_k^\rr = \norm{\hat{\Gamma}_k^\rr(\cdot, \eps) - 
  \hat{\Gamma}_k^\rd(\cdot, \eps)}_{0, \infty}$. Notice that when $k 
  = 0, 1, \ldots, \lfloor \frac{K-1}{2} \rfloor$, 
  $\nabla_x \hat{\Gamma}_{k}^\rd(x,\eps)$ is bounded by \cref{thm:nd}.
  One can obtain by \cref{remark:monotone} that 
  \begin{displaymath}
    \begin{aligned}
      & \left| \hat{\Gamma}_{k+1}^\rr(x, \eps) - 
      \hat{\Gamma}_{k+1}^\rd(x, \eps) \right| \\
      \le \; & \Bigg| \hat{\Gamma}_{k+1}^\rr(x,\eps) - 
      \tilde{g}_x^{-1} \Bigg( 
      \eps \frac{\hat{\Gamma}_k^\rr(x + f(x, 
        \hat{\Gamma}_k^\rr(x,\eps)) \tau, \eps) - 
        \hat{\Gamma}_k^\rr(x,\eps)}{\tau} \Bigg) \Bigg| \\
      & \quad + \Bigg| \tilde{g}_x^{-1} 
      \Bigg( \eps \frac{\hat{\Gamma}_k^\rr(x + f(x, 
        \hat{\Gamma}_k^\rr(x,\eps)) \tau, \eps) - 
        \hat{\Gamma}_k^\rr(x,\eps)}{\tau} \Bigg) \\
      & \qquad\qquad - \tilde{g}_x^{-1} 
      \Bigg( \eps \frac{\hat{\Gamma}_k^\rd(x + f(x, 
      \hat{\Gamma}_k^\rr(x,\eps)) \tau, \eps) - 
      \hat{\Gamma}_k^\rd(x,\eps)}{\tau} \Bigg)\Bigg| \\
      & \quad + \Bigg| \tilde{g}_x^{-1} 
      \Bigg( \eps \frac{\hat{\Gamma}_k^\rd(x + f(x, 
        \hat{\Gamma}_k^\rr(x,\eps)) \tau, \eps) - 
        \hat{\Gamma}_k^\rd(x,\eps)}{\tau} \Bigg) \\
      & \qquad\qquad - \tilde{g}_x^{-1} 
      \Bigg( \eps \frac{\hat{\Gamma}_k^\rd(x + f(x, 
      \hat{\Gamma}_k^\rd(x,\eps)) \tau, \eps) - 
      \hat{\Gamma}_k^\rd(x,\eps)}{\tau} \Bigg) \Bigg| \\
      \le \; & \mathcal{O}\left(\eta_{k+1} + \frac{\eps}{\tau} 
      e_k^\rr \right).
    \end{aligned}
  \end{displaymath}
  Therefore, $e_{k+1}^\rr = \mathcal{O}\left( \eta_{k+1} 
  +\frac{\eps}{\tau} e_k^\rr \right)$, and then the proof is 
  completed by induction.
\end{proof}
\begin{remark}\label{remark:accum}
  For fixed $k$, let $\eta = \max\limits_{j = 0, 1, \ldots, k} 
  \eta_j$.
  By \cref{thm:accum}, if $\eps = \mathcal{O}(\tau)$, then 
  $$\norm{\hat{\Gamma}_k^\rr(\cdot, \eps) - 
    \hat{\Gamma}_k^\rd(\cdot, \eps)}_{0, \infty} = \mathcal{O} 
  \left( \eta \right);$$ if $\tau = o(\eps)$, then 
  $$\norm{\hat{\Gamma}_k^\rr(\cdot, \eps) - \hat{\Gamma}_k^\rd(\cdot, 
  \eps)}_{0, \infty} = \mathcal{O} \left( \frac{\eps^k}{\tau^k} \eta 
  \right).$$
\end{remark}

\subsection{Global error}
In \Cref{sec:analysis_nd,sec:ana_accum}, we have analyzed the 
numerical error when calculating $\Gamma_k(x,\eps)$. Now we would 
like to figure out how it influences the global error. Let 
\begin{displaymath}
  \kappa = \norm{\Gamma_k(\cdot, \eps) - \hg_k(\cdot, \eps)}_{0, 
  \infty} < +\infty.
\end{displaymath}
We would like to compare the solutions of the following two ODEs:
\begin{displaymath}
  \od{X_k}{t} = f(X_k, \Gamma_k(X_k, \eps)),
\end{displaymath}
and 
\begin{displaymath}
  \od{\hat{X}_k}{t} = f(\hat{X}_k, \hg_k(\hat{X}_k, \eps)).
\end{displaymath}
By Gronwall's inequality and the boundedness of $\nabla_x 
\Gamma_k(x,\eps)$ when $k = 0,1,\ldots,K$ (See 
\cref{lemma:iter}), one can deduce the following proposition. By 
\cref{prop:global}, one obtains that the order of global error 
between $X_k(t)$ and $\hat{X}_k(t)$ in a finite time horizon is the 
same as order of $\ml^\infty$-error between $\Gamma_k(\cdot, \eps)$ 
and $\hg_k(\cdot, \eps)$, if $X_k(0) = \hat{X}_k(0)$.
\begin{proposition}\label{prop:global}
  For $k = 0, 1, \ldots, K$, there exists a constant $C > 0$ such 
  that 
  \begin{displaymath}
    |X_k(t) - \hat{X}_k(t)|^2 \le e^{Ct} \left( |X_k(0) - 
    \hat{X}_k(0)|^2 + \kappa^2 \right).
  \end{displaymath}
\end{proposition}

%% file: article_numerical.tex
\section{Numerical Experiments}
\label{sec:num}
In this section, we perform numerical simulations on several 
examples 
to demonstrate numerical efficiency of our models and algorithms 
developed in the previous sections.

The numerical experiments are set up as follows. We conduct the 
simulations in the time interval $[0,T]$. We use the coupled solver 
in the interval $[0,T_c]$, where $T_c$ is determined by the criteria 
in \cref{remark:terminal}. 
The coupled solver utilizes the common $4$th-order explicit 
Runge-Kutta scheme (RK4) with time step $\Delta t_c$. The 
HMM-type algorithms are used in the interval $[T_c,T]$. The 
macroscopic solver utilizes RK4 with time step size 
$\Delta t$. 
The microscopic solver 
utilizes the Forward Euler scheme (FE) with number of steps $M$ and 
the time step size $\delta t=\alpha \eps$, where $\alpha$ is a 
given 
constant dependent on the problem. 
We compare the $\ell_2$-norm error of slow variable, that is, 
$|x(T) 
- x_N|$.
In all numerical experiments, the initial values $(x_0,y_0)$ and 
the 
terminal time $T$ are specified to ensure the dissipativity of 
the 
system on $[0,T]$.

\subsection{A naive example}
In this part, we use a naive example to test the accuracy and 
efficiency of our algorithms.
\begin{example}\label{exp:naive}
  Let us consider the following example:
  \begin{displaymath}
    \left\{
    \begin{aligned}
      & \od{x}{t} = y, \\
      & \od{y}{t} = \frac{1}{\eps} (x - y), \\
      & x|_{t = 0} = x_0, \ y|_{t = 0} = y_0,
    \end{aligned}
    \right.
  \end{displaymath}
  with $x_0 = 1$, $y_0 = 2$, $T = 4$. The exact solution of this 
  example is 
  \begin{displaymath}
    x(t) = \frac{-\lambda_2 x_0 + y_0}{\lambda_1 - \lambda_2} 
    e^{\lambda_1 t} + \frac{\lambda_1 x_0 - y_0}{\lambda_1 - 
    \lambda_2} e^{\lambda_2 t},
  \end{displaymath}
  where $\lambda_1 = - \frac{1 + \sqrt{1 + 4\eps}}{2 \eps}$ and 
  $\lambda_2 = - \frac{1 - \sqrt{1 + 4\eps}}{2 \eps}$.
  
  Parameters: $\eps = 1.0 \times 10^{-5}$, $\tau = 1.0 \times 
  10^{-5}$, 
  $M = 1$, $\alpha = 1.0$, $\Delta t_c = 1.0 \times 10^{-5}$, $\Delta 
  t = 5.0 \times 10^{-3}$, $\hat{\beta} = 1$.
  
  \cref{alg1} and the forward difference formula are adopted. We 
  compare our HMM$k$ solvers with the 
  coupled solver, with respect to the numerical error and the total 
  computing time. To ensure fairness when comparing the 
  accuracy of different HMM-type algorithms, we use the same $T_c$. 
  In other words, we take $k = 2$ in \cref{remark:terminal} for all 
  the HMM-type algorithms. The numerical results are reported in 
  \cref{tab:exp_naive}. 
\end{example}

\begin{table}[h]
  \centering
  \caption{Numerical results for \cref{exp:naive}.}
  \label{tab:exp_naive}
  \begin{tabular}{l|ccc}
    \hline\hline
    solver & error & time (s) & $T_c$ \\\hline
    coupled & 2.1832e-09 & 5.44 & 4.0e+00 \\
    HMM0 & 2.1836e-03 & 0.02 & 4.0e-04 \\
    HMM1 & 4.6017e-08 & 0.04 & 4.0e-04 \\
    HMM2 & 2.3441e-09 & 0.09 & 4.0e-04 \\
    \hline\hline
  \end{tabular}
\end{table}

For one thing, our high order numerical homogenization method 
reduces the numerical error, compared with the classical HMM$0$ 
solver. For another thing, the HMM$2$ solver uses far less time 
than the coupled solver to reach roughly the same error. In this 
numerical example, the time cost of HMM-type algorithms is mainly 
due to the macroscopic solver, since here it is easy to compute 
$\gamma(x)$ by the microscopic solver. By this numerical 
experiment, we exhibit the accuracy and efficiency of out method.

\subsection{Model validation}
In the previous section, we have already proven that the modeling 
error of HMM$k$ method turns out to be $\mathcal{O}(\eps^{k+1})$. 
In 
this part, we would like to test and verify this modeling error 
estimate by numerical experiments. Let us consider the following 
three numerical examples. For each numerical example, we compare 
numerical error between numerical solutions and the reference 
solution.

\begin{example}\label{exp:enzyme}
  Enzyme reaction equation 
  \cite{Carr1991Applications,Heineken1967On}:
  \begin{displaymath}
    \left\{
    \begin{aligned}
      &\od{x}{t} = -x + (x+c)y, \\
      &\od{y}{t} = \frac{1}{\eps} \left( x - (x+1)y \right), \\
      & x|_{t = 0} = x_0, \ y|_{t = 0} = y_0,
    \end{aligned}
    \right.
  \end{displaymath}
  with $x_0 = 1$, $y_0 = 0$, $c = 0.5$, $T = 1$.
  
  Parameters: $\tau = \e{1.0}{-6}$, $M = 10$, $\alpha = 0.5$, 
  $\Delta t_c = \e{1.0}{-5}$, $\hat{\beta} = 1.5$.
  
  \cref{alg1} and the central difference formula are adopted. The 
  reference solution is given by the 
  coupled solver with time step size $\Delta \tilde{t}_c = 
  1.00\times10^{-6}$.
  See \cref{fig2_1} for fixed $\Delta t = \e{1.0}{-2}$ and different 
  $\eps$.
  See \cref{fig2_2} for fixed $\eps = \e{1.0}{-2}$ and different 
  $\Delta t$.
\end{example}

\begin{figure}[p]
  \centering
  \subfigure[Different $\eps$ and fixed $\Delta t$]{
    \includegraphics[width=0.45\linewidth]{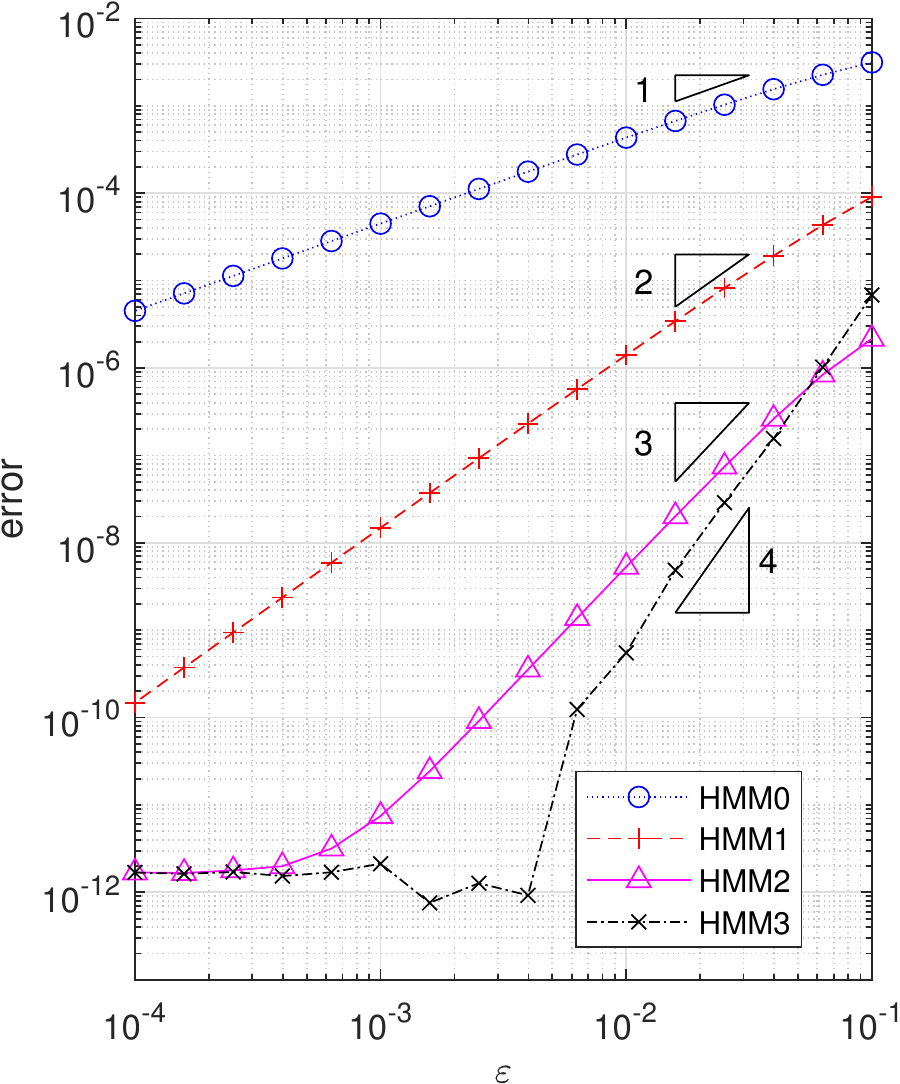}
    \label{fig2_1}
  }
  \quad
  \subfigure[Different $\Delta t$ and fixed $\eps$]{
    \includegraphics[width=0.45\linewidth]{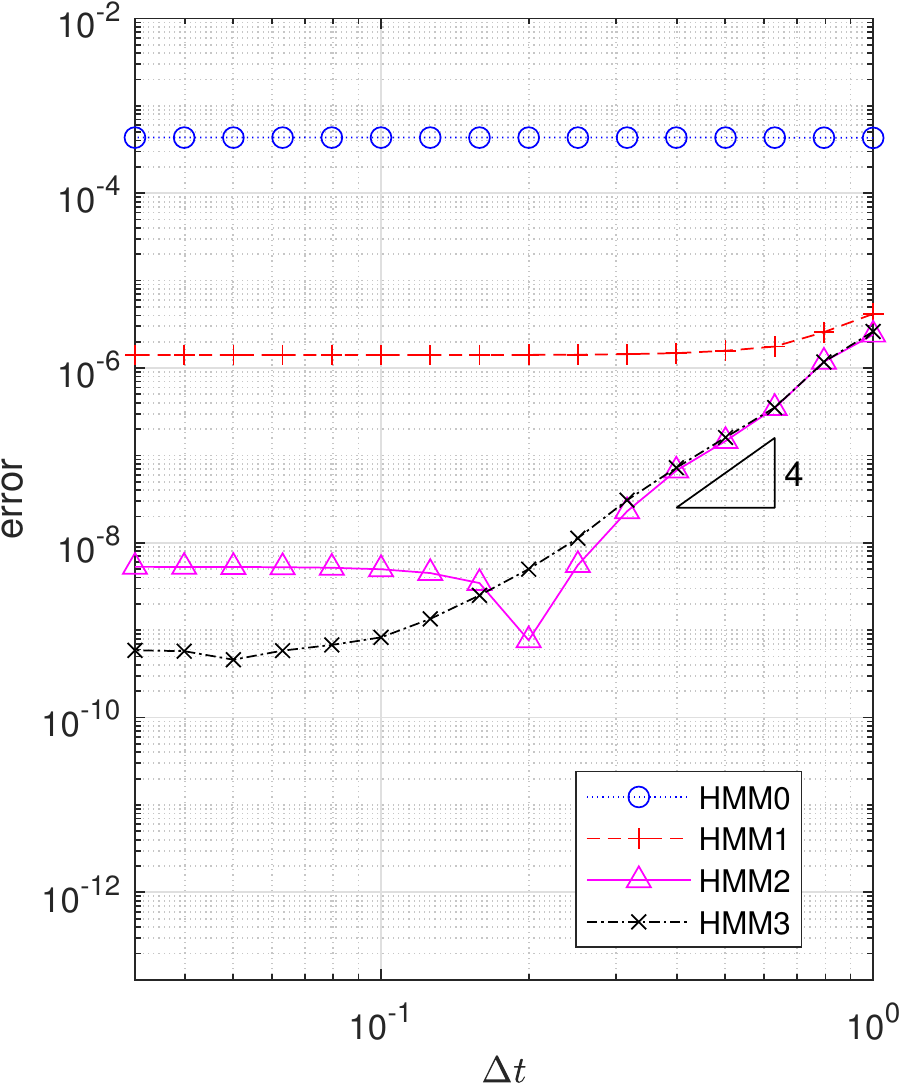}
    \label{fig2_2}
  }
  \caption{Numerical results for \cref{exp:enzyme}.}
  \label{fig2}
\end{figure}

\begin{example}\label{exp:fvdp}
  Forced Van der Pol equation \cite{Guckenheimer2003The}:
  \begin{displaymath}
    \left\{
    \begin{aligned}
      & \od{x^{(1)}}{t} = -y + a\sin(2\pi x^{(2)}),\\
      & \od{x^{(2)}}{t} = b,\\
      & \od{y}{t} = \frac{1}{\eps} \left( y + x^{(1)} - 
      \frac{1}{3} 
      y^3 \right), \\
      & x|_{t = 0} = (x^{(1)}, x^{(2)})|_{t = 0} = x_0, \ y|_{t = 
      0} 
      = y_0,
    \end{aligned}
    \right.
  \end{displaymath}
  with $x_0=(3,1)$, $y_0=1$, $a=2$, $b=1$, $T=1$.
  
  Parameters: $\tau=\e{1.0}{-6}$, $M=25$, $\alpha=0.1$, 
  $\Delta t_c = \e{1.0}{-5}$, $\hat{\beta} = 0.01$.
  
  \cref{alg2} and the forward difference formula are adopted. The 
  reference solution is given by the 
  coupled solver with time step size $\Delta t_c$.
  See \cref{fig5_1} for fixed $\Delta t = \e{1.0}{-2}$ and different 
  $\eps$.
  See \cref{fig5_2} for fixed $\eps=\e{1.0}{-4}$ and different 
  $\Delta t$.
\end{example}

\begin{figure}[p]
  \centering
  \subfigure[Different $\eps$ and fixed $\Delta t$]{
    \includegraphics[width=0.45\linewidth]{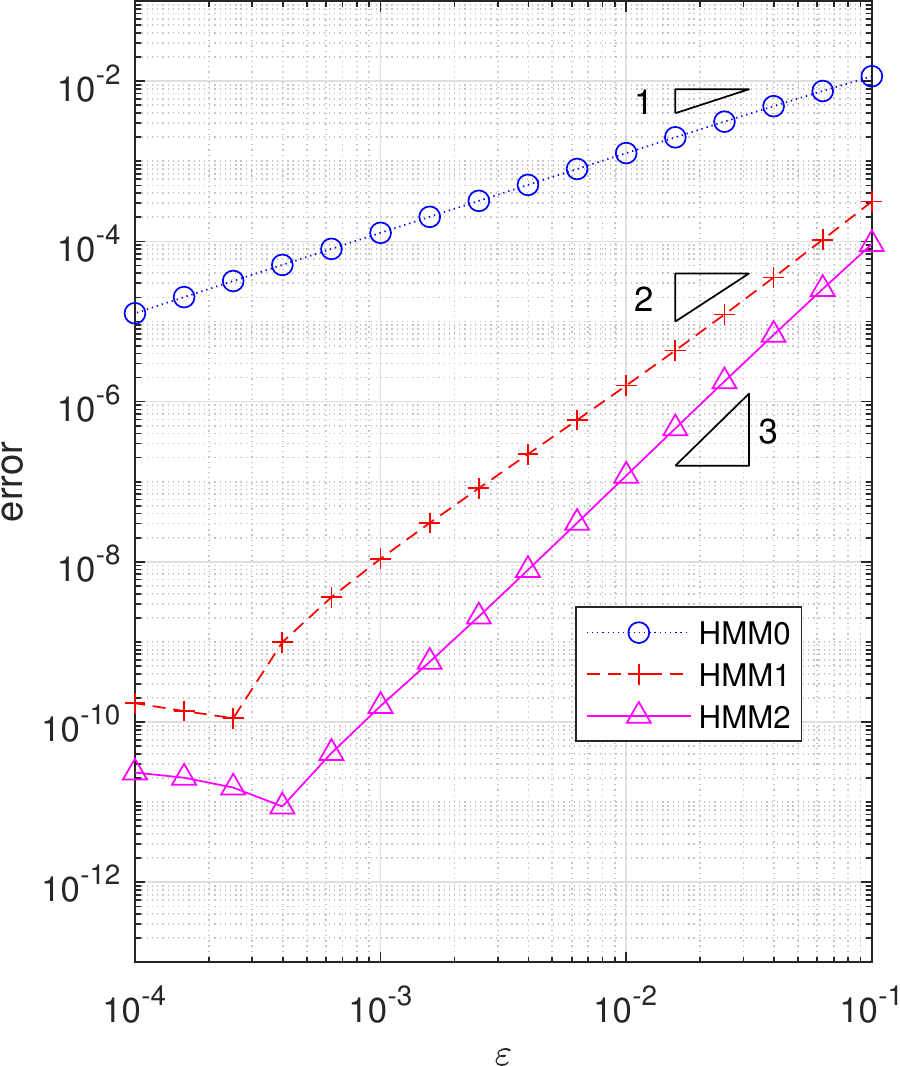}
    \label{fig5_1}
  }
  \quad
  \subfigure[Different $\Delta t$ and fixed $\eps$]{
    \includegraphics[width=0.45\linewidth]{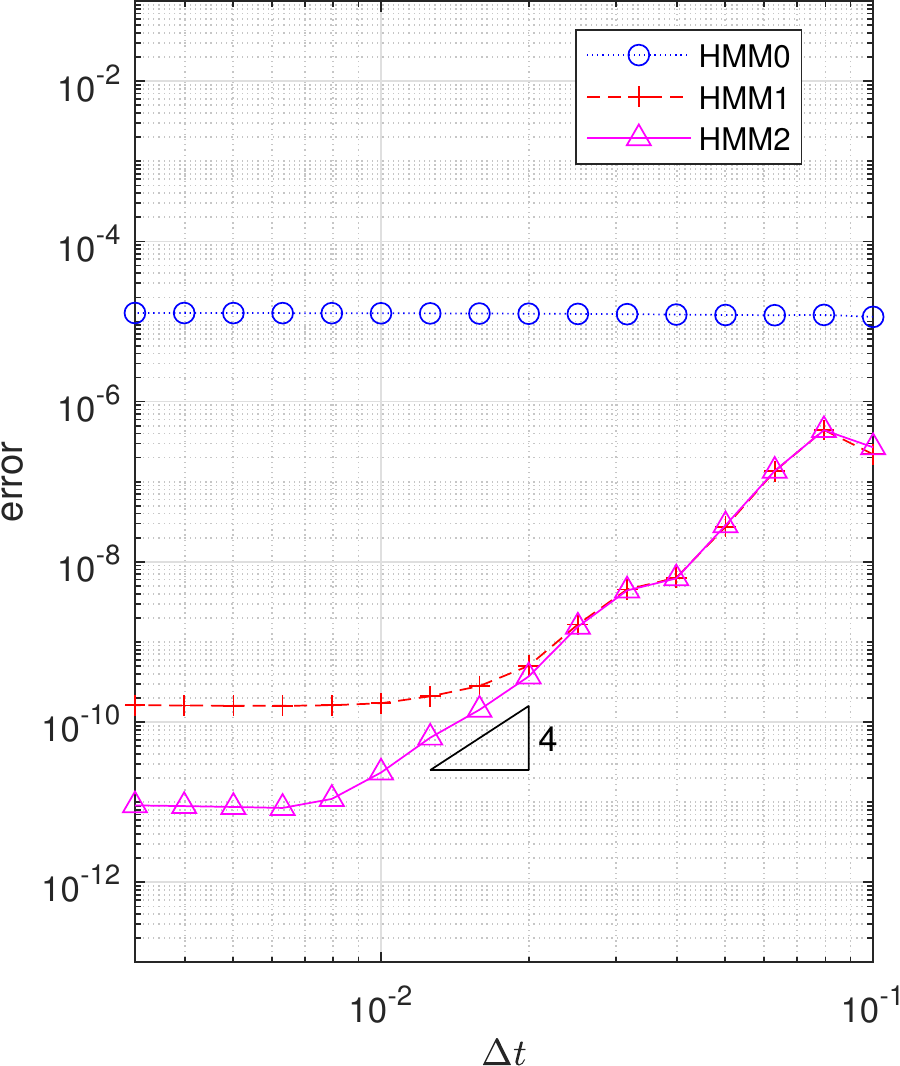}
    \label{fig5_2}
  }
  \caption{Numerical results for \cref{exp:fvdp}.}
  \label{fig5}
\end{figure}

\begin{example}
  \label{exp:6}
  Cubic Chua's model \cite{Chua1986The}:
  \begin{displaymath}
    \left\{
    \begin{aligned}
      & \od{x^{(1)}}{t} = -d x^{(2)}, \\
      & \od{x^{(2)}}{t} = -a y + x^{(1)} + b x^{(2)}, \\
      & \od{y}{t} = \frac{1}{\eps} (x^{(2)} -c_3 y^3 - c_2 y^2 - 
      c_1 
      y), \\
      & x|_{t = 0} = (x^{(1)}, x^{(2)})|_{t = 0} = x_0, \ y|_{t = 
      0} 
      = y_0,
    \end{aligned}
    \right.
  \end{displaymath}
  with $x_0 = (1,1)$, $y_0=1$, $a=0.7$, $b=0.25$, $c_1=7$, $c_2=15$, 
  $c_3=20$, $d =1$, $T=1$.
  
  Parameters: $\tau=\e{1.0}{-6}$, $M=10$, $\alpha=0.1$, 
  $\Delta t_c = \e{1.0}{-6}$, $\hat{\beta} = 10$.
  
  \cref{alg2} and the central difference formula are adopted. The 
  reference solution is given by the 
  coupled solver with time step size $\Delta t_c$.
  See \cref{fig6_1} for fixed $\Delta t = \e{1.0}{-2}$ and different 
  $\eps$.
  See \cref{fig6_2} for fixed $\eps = \e{1.0}{-2}$ and different 
  $\Delta t$.
\end{example}

\begin{figure}[p]
  \centering
  \subfigure[Different $\eps$ and fixed $\Delta t$]{
    \includegraphics[width=0.45\linewidth]{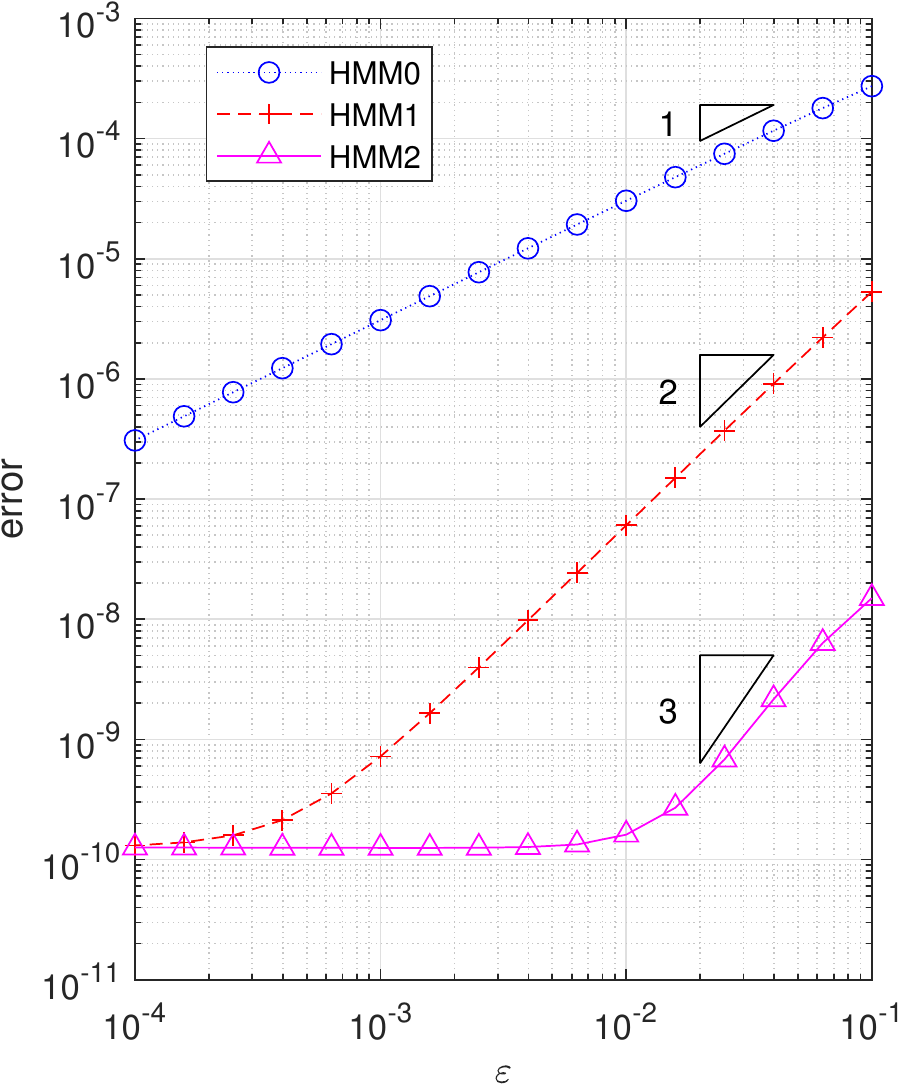}
    \label{fig6_1}
  }
  \quad
  \subfigure[Different $\Delta t$ and fixed $\eps$]{
    \includegraphics[width=0.45\linewidth]{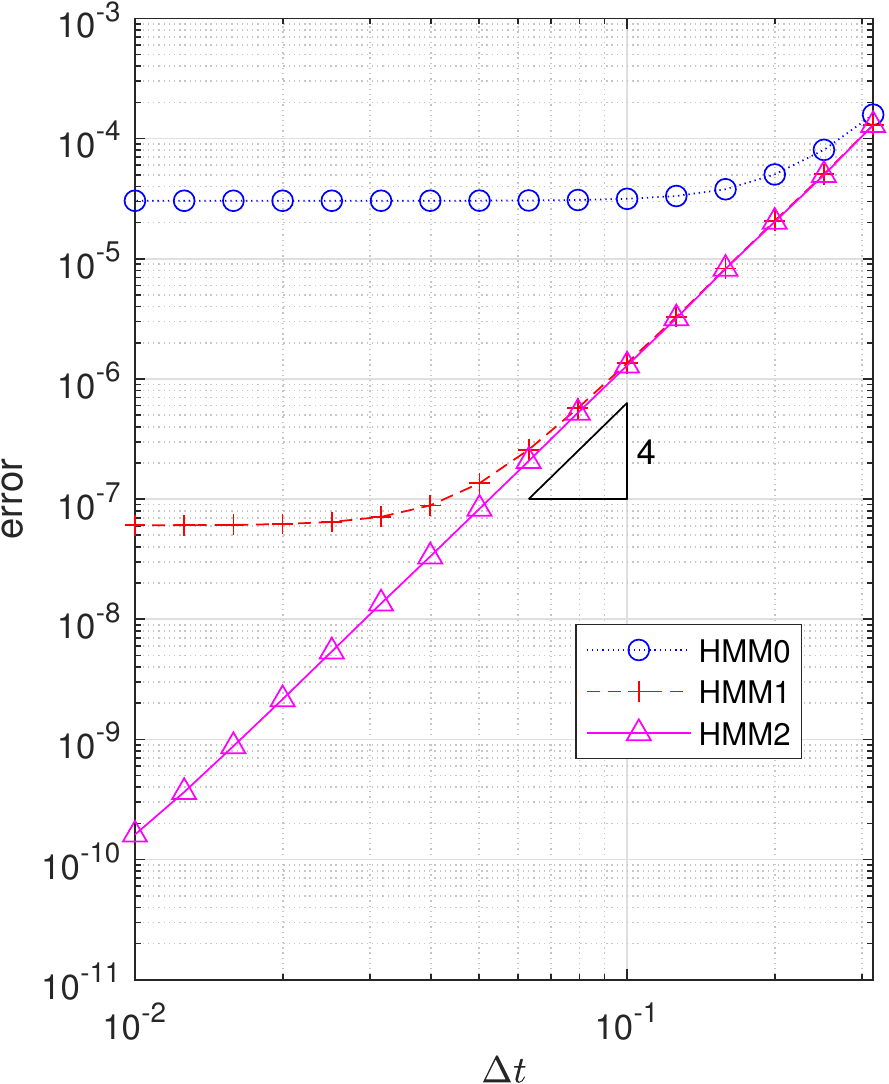}
    \label{fig6_2}
  }
  \caption{Numerical results for \cref{exp:6}.}
  \label{fig6}
\end{figure}

The results in \cref{fig2_1,fig5_1,fig6_1} show the order 
of numerical error. By numerical investigation, the theoretical 
result that the modeling error of HMM$k$ is of order 
$\mathcal{O}(\eps^{k+1})$ is verified. The results in 
\cref{fig2_2,fig5_2,fig6_2} indicate whether the modeling error or 
the truncation error of macroscopic solver takes a dominant position 
in the numerical error.


\subsection{Numerical derivative}
It has been proven that the sampling error generated by numerical 
derivatives is $\mathcal{O}(\eps \tau)$, if the forward difference 
formula is adopted. As seen in \cref{remark:hd}, when the central 
difference formula is adopted, this part of error turns out to be 
$\mathcal{O}(\eps \tau^2)$. Let us consider the following numerical 
example.

\begin{example}\label{exp:vdp}
  Van der Pol equation 
  \cite{Hairer1980Solving,Eriksson2012Explicit}:
  \begin{displaymath}
    \left\{
    \begin{aligned}
      & \od{x}{t} = y,\\
      & \od{y}{t} = -\frac{1}{\eps} \left( (x^2 - 1) y + x 
      \right), \\
      & x|_{t=0}=x_0, \ y|_{t=0}=y_0,
    \end{aligned}
    \right.
  \end{displaymath}
  with $x_0=4$, $y_0=2$, $T=5$.
  
  Parameters: $M=20$, $\alpha=0.1$, 
  $\Delta t_c = 1.0\times10^{-5}$, $\Delta t = \e{2.0}{-2}$, 
  $\hat{\beta} = 3$.
  
  \cref{alg2} is adopted. The reference solution is given by the 
  coupled solver with time step size $\Delta t_c$. We test the 
  numerical error for different $\tau$ and different $\eps$.
  See \cref{fig:tau1_1} for the forward difference formula.
  See \cref{fig:tau1_2} for the central difference formula.
\end{example}

\begin{figure}[p]
  \centering
  \subfigure[Forward difference formula]{
    \includegraphics[width=0.45\linewidth]{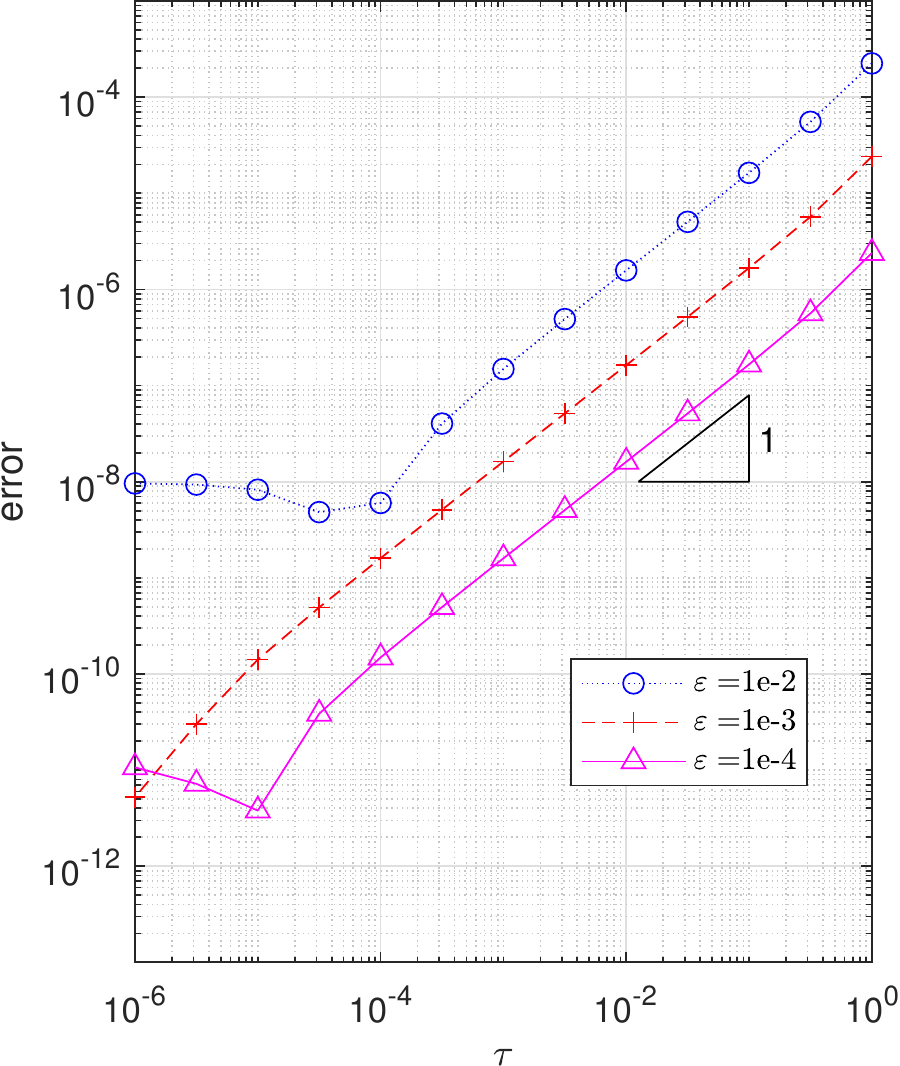}
    \label{fig:tau1_1}
  }
  \quad
  \subfigure[Central difference formula]{
    \includegraphics[width=0.45\linewidth]{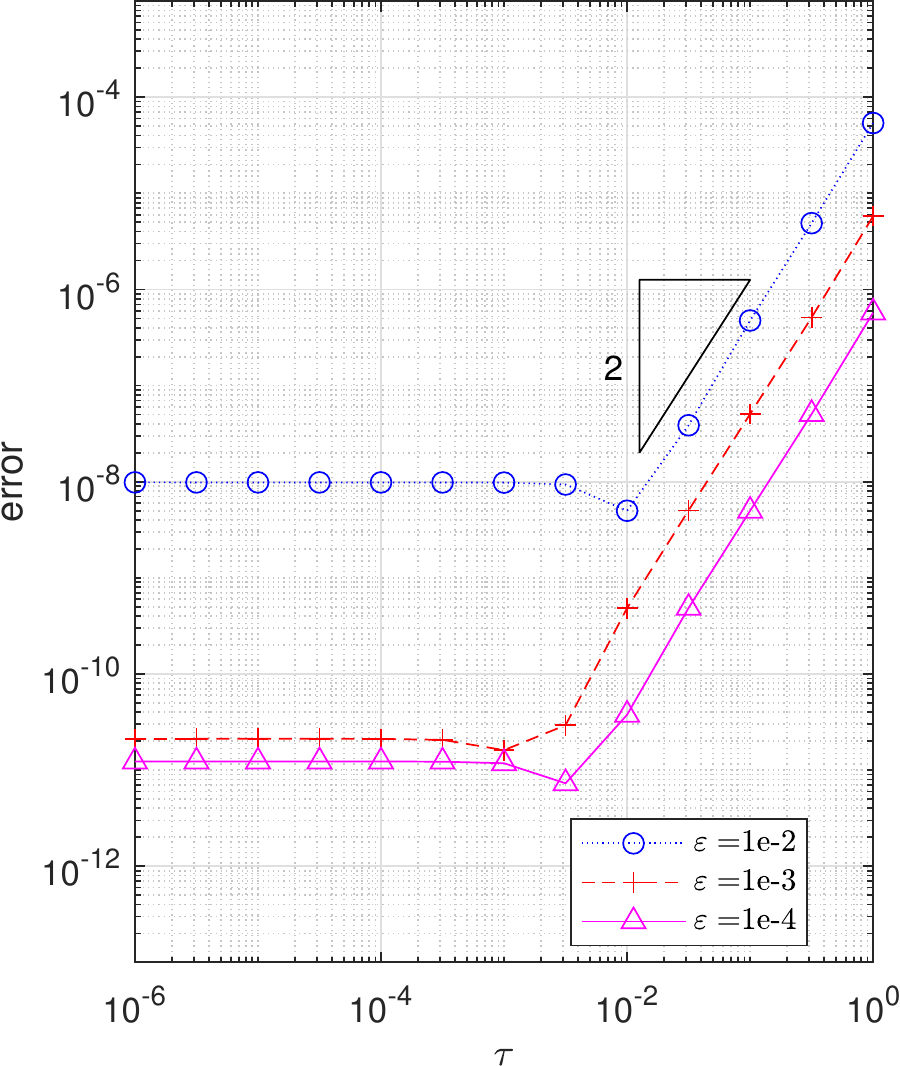}
    \label{fig:tau1_2}
  }
  \caption{Numerical results for \cref{exp:vdp}.}
  \label{fig:tau1}
\end{figure}

From the numerical results, one can see that, when $\tau$ is 
relatively large, the numerical error is approximately 
$\mathcal{O}(\eps \tau)$ for the forward difference formula, and 
$\mathcal{O}(\eps \tau^2)$ for the central difference formula.
Therefore, the theoretical analysis in \Cref{sec:analysis_nd} is 
verified.

%% file: article_conclusion.tex
\section{Conclusions}
\label{sec:concl}
We proposed a high-order numerical homogenization method for the dissipative ordinary 
differential equations. We develop the correction models based on the asymptotic 
approximations and a novel iterative formula. The corresponding numerical algorithms 
are designed in the framework of the heterogeneous multiscale methods. We provide some 
theoretical analysis on our algorithms. By numerical investigation, not only the error 
estimates are verified, but also the efficiency of our methods is exhibited.

\section*{Acknowledgements}
Zeyu Jin is supported by the Elite Undergraduate Training Program of
School of Mathematical Sciences in Peking University. Ruo Li is
partially supported by the National Key R\&D Program of China
(No. 2020YFA0712000) and the National Science Foundation in China (No.
11971041).

%% file: article_appendix.tex
\appendix
\renewcommand{\appendixname}{Appendix}

\section{Well-posedness of \cref{equ:center_linear} for Small $\eps$}
\label{app:linear}
\begin{theorem}
	\label{thm:linear}
	Under the assumptions in \cref{example:linear_1}, for each 
	$M>\norm{A_{22}^{-1}A_{21}}$, there exists $\delta\in(0,\eps_0]$, 
	such that for each $\eps\in(0,\delta]$, there exists a unique 
	solution $(C^*,d^*)$ to \cref{equ:center_linear} such that 
	$\norm{C^*}\le M$. Furthermore, 
	$C^*+A_{22}^{-1}A_{21}=\mathcal{O}(\eps)$ and 
	$d^*+A_{22}^{-1}b_2=\mathcal{O}(\eps)$.
\end{theorem}
\begin{proof}
	First, we consider the equation \cref{equ:center_linear_1}. It is 
	easy to see that the solution of \cref{equ:center_linear_1} is the 
	fixed point of the mapping
	\begin{displaymath}
	\mathcal{T}:C\mapsto -A_{22}^{-1}A_{21}+\eps 
	A_{22}^{-1}CA_{11}+\eps A_{22}^{-1}CA_{12}C.
	\end{displaymath}
	We define a set $\mathcal{E}=\{C\in\mathbb{R}^{n_y\times 
	n_x}:\norm{C}\le M\}$, where $M>\norm{A_{22}^{-1}A_{21}}$ is 
	arbitrary. Let 
	$M_1=\norm{A_{22}^{-1}}(\norm{A_{11}}+\norm{A_{12}})$.
	When $C\in\mathcal{E}$, we have that 
	\begin{displaymath}
	\begin{aligned}
	\norm{\mathcal{T}C}\le\norm{A_{22}^{-1}A_{21}}+\eps M_1 \norm{C} 
	+\eps M_1\norm{C}^2
	\le \norm{A_{22}^{-1}A_{21}}+\eps M_1 (M+M^2).
	\end{aligned}
	\end{displaymath}
	We take 
	$\delta_1 = \min \left\{ 
	\eps_0,\frac{M-\norm{A_{22}^{-1}A_{21}}}{M_1(M+M^2)} \right\}>0$.
	For each $\eps\in(0,\delta_1]$, we have that 
	$\mathcal{T}:\mathcal{E}\rightarrow\mathcal{E}$. For each 
	$C,\tilde{C}\in\mathcal{E}$, 
	\begin{displaymath}
	\begin{aligned}
	\norm{\mathcal{T}C-\mathcal{T}\tilde{C}}&=\norm{\eps 
	A_{22}^{-1}(C-\tilde{C})A_{11}+\eps A_{22}^{-1}(C-\tilde{C})A_{12} 
	C +\eps A_{22}^{-1} \tilde{C} A_{12} (C-\tilde{C})}\\
	&\le \eps M_1(1+2M) \norm{C-\tilde{C}}.
	\end{aligned}
	\end{displaymath}
	We take $\delta_2 =\min\{\delta_1,\frac{1}{2M_1(1+2M)}\}>0$. Then 
	for each $\eps\in(0,\delta_2]$, $\mathcal{T}$ is a contraction 
	mapping on $\mathcal{E}$. Therefore, there exists a unique solution 
	$C^*$ to \cref{equ:center_linear_1} in $\mathcal{E}$. We take 
	$\delta=\min\{\delta_2,\frac{1}{2M_1M}\}$. For each 
	$\eps\in(0,\delta]$, $\norm{\eps A_{22}^{-1}C^*A_{12}}\le \eps 
	M_1M\le\frac{1}{2}<1$. Thus, $A_{22}-\eps C^*A_{12}$ is invertible. 
	With $C^*$ fixed, there exists a unique solution $d^*$ to 
	\cref{equ:center_linear_2} in $\ry$, and $d^*$ is uniformly 
	bounded in $\eps\in(0,\delta]$. Furthermore, we notice that 
	\begin{displaymath}
	\begin{aligned}
	&C^*+A_{22}^{-1}A_{21}=\eps(A_{22}^{-1}C^*A_{11}+A_{22}^{-1}C^*A_{12}C^*)=\mathcal{O}(\eps),\\
	&d^*+A_{22}^{-1}b_2=\eps(A_{22}^{-1}C^*A_{12}d^*+A_{22}^{-1}C^*b_1)=\mathcal{O}(\eps).
	\end{aligned}
	\end{displaymath}
	Then the proof is completed.
\end{proof}

\section{Proof of \cref{thm:iteration}}
\label{app:pf_iter}
\begin{proof}[Proof of \cref{thm:iteration}]
  By \cref{remark:monotone}, one can obtain that
  \begin{displaymath}
    \begin{aligned}
      |\Gamma_1(x,\eps)-\gamma(x)-\eps\gamma_1(x)| & \le 
      \frac{1}{\beta} |g(x,\Gamma_1(x,\eps)) - 
      g(x,\gamma(x)+\eps\gamma_1(x))| \\
      & = \frac{1}{\beta}| \eps 
      \nabla \gamma(x) f(x,\gamma(x)) - 
      g(x,\gamma(x)+\eps\gamma_1(x)) |.
    \end{aligned}
  \end{displaymath}
  By \cref{equ:expression_pre_1} and Taylor's expansion of 
  $g(x,y)$ at $(x,\gamma(x))$, 
  \begin{displaymath}
    \begin{aligned}
      &\eps\nabla\gamma(x)f(x,\gamma(x))-g(x,\gamma(x)+\eps\gamma_1(x))\\
      =\, &\eps G_y(x)\gamma_1(x)-\big(g(x,\gamma(x))+\eps 
      G_y(x)\gamma_1(x)+\mathcal{O}(\eps^2)\big)\\
      =\, &\mathcal{O}(\eps^2),
    \end{aligned}
  \end{displaymath}
  where $\mathcal{O}(\eps^2)$ can be controlled uniformly thanks to 
  \cref{assump:bound}. In other words, there exists a 
  constant $C_1>0$ such that 
  $|\eps\nabla\gamma(x)f(x,\gamma(x))-g(x,\gamma(x)+\eps\gamma_1(x))
  |\le C_1\eps^2$. Therefore, 
  \begin{displaymath}
    |\Gamma_1(x,\eps)-\gamma(x)-\eps\gamma_1(x)|\le 
    \frac{C_1}{\beta}\eps^2,
  \end{displaymath}
  which completes the proof of \cref{thm:iteration_1}.
  
  By the expression of $\gamma_1(x)$ in 
  \cref{equ:expression_pre_1}, one gets that 
  \begin{equation}
    \label{equ:thm_iteration_2_1}
    g(x,\Gamma_1(x,\eps))=\eps G_y(x)\gamma_1(x).
  \end{equation}
  For the sake of clarity, we consider the $i$-th component of 
  \cref{equ:thm_iteration_2_1}, that is,
  \begin{equation}
    \label{equ:thm_iteration_2_2}
    g^i(x,\Gamma_1(x,\eps))=\eps\sum_{k=1}^{n_y}\frac{\partial 
    g^i}{\partial y^k}(x,\gamma(x))\gamma_1^k(x).
  \end{equation}
  Taking the derivative of \cref{equ:thm_iteration_2_2} with 
  respect to $x^j$, we obtain that 
  \begin{equation}
    \label{equ:thm_iteration_2_3}
    \begin{aligned}
      &\frac{\partial g^i}{\partial 
      x^j}(x,\Gamma_1(x,\eps))+\sum_{k=1}^{n_y}\frac{\partial 
      g^i}{\partial y^k}(x,\Gamma_1(x,\eps))\frac{\partial 
      \Gamma_1^k}{\partial x^j}(x,\eps)\\
      =\ &\eps\sum_{k=1}^{n_y}\bigg( \frac{\partial^2 g^i}{\partial 
      y^k \partial x^j}(x,\gamma(x)) + 
      \sum_{\ell=1}^{n_y}\frac{\partial^2 g^i}{\partial y^k \partial 
      y^\ell}(x,\gamma(x))\frac{\partial \gamma^\ell}{\partial 
      x^j}(x) \bigg)\gamma_1^k(x)\\
      &\quad\quad +\eps\sum_{k=1}^{n_y}\frac{\partial g^i}{\partial 
      y^k}(x,\gamma(x))\frac{\partial \gamma_1^k}{\partial x^j}(x).
    \end{aligned}
  \end{equation}
  Actually \cref{equ:thm_iteration_2_3} yields that 
  $\nabla_x\Gamma_1(x,\eps)$ is uniformly bounded in 
  $\eps\in(0,\eps_0]$, since
  \begin{displaymath}
    \nabla_x\Gamma_1(x,\eps)=G_y(x)^{-1}A,
  \end{displaymath}
  where $A=A(x,\eps)=[A^{i,j}]_{i=1,j=1}^{n_y,n_x}$, $A^{i,j}$ is 
  defined as
  \begin{displaymath}
    \begin{aligned}
      A^{i,j}&=\eps\sum_{k=1}^{n_y}\bigg( \frac{\partial^2 
      g^i}{\partial y^k \partial x^j}(x,\gamma(x)) + 
      \sum_{\ell=1}^{n_y}\frac{\partial^2 g^i}{\partial y^k \partial 
      y^\ell}(x,\gamma(x))\frac{\partial \gamma^\ell}{\partial 
      x^j}(x) \bigg)\gamma_1^k(x)\\
      &\quad\quad\quad+\eps\sum_{k=1}^{n_y}\frac{\partial 
      g^i}{\partial y^k}(x,\gamma(x))\frac{\partial 
      \gamma_1^k}{\partial x^j}(x)-\frac{\partial g^i}{\partial 
      x^j}(x,\Gamma_1(x,\eps)),
    \end{aligned}
  \end{displaymath}
  and it is obvious that $A(x,\eps)$ is uniformly bounded in 
  $\eps\in(0,\eps_0]$ due to \cref{assump:bound}. By 
  \cref{thm:iteration_1} and Taylor's expansion of 
  $\frac{\partial g^i}{\partial x^j}$ and 
  $\frac{\partial g^i}{\partial y^k}$ at $(x,\gamma(x))$, one obtains 
  that 
  \begin{equation}
    \label{equ:thm_iteration_2_5}
    \frac{\partial g^i}{\partial 
    x^j}(x,\Gamma_1(x,\eps))=\frac{\partial g^i}{\partial 
    x^j}(x,\gamma(x))+\eps\sum_{k=1}^{n_y}\frac{\partial^2 
    g^i}{\partial x^j \partial 
    y^k}(x,\gamma(x))\gamma_1^k(x)+\mathcal{O}(\eps^2),
  \end{equation}
  and 
  \begin{equation}
    \label{equ:thm_iteration_2_6}
    \frac{\partial g^i}{\partial 
    y^k}(x,\Gamma_1(x,\eps))=\frac{\partial g^i}{\partial 
    y^k}(x,\gamma(x))+\eps\sum_{\ell=1}^{n_y}\frac{\partial^2 
    g^i}{\partial y^k\partial 
    y^\ell}(x,\gamma(x))\gamma_1^\ell(x)+\mathcal{O}(\eps^2),
  \end{equation}
  where both $\mathcal{O}(\eps^2)$ can be controlled uniformly due 
  to \cref{assump:bound}. In addition, the 
  $(i,j)$-component of \cref{equ:nabla_gamma} can be written as
  \begin{equation}
    \label{equ:thm_iteration_2_7}
    \frac{\partial g^i}{\partial 
    x^j}(x,\gamma(x))=-\sum_{k=1}^{n_y}\frac{\partial g^i}{\partial 
    y^k}(x,\gamma(x))\frac{\partial \gamma^k}{\partial x^j}(x).
  \end{equation}	
  By \cref{equ:thm_iteration_2_3}, \cref{equ:thm_iteration_2_5}, 
  \cref{equ:thm_iteration_2_6}, \cref{equ:thm_iteration_2_7} and 
  the uniform boundedness of $\nabla_x \Gamma_1(x,\eps)$, one can 
  obtain that
  \begin{equation}
    \label{equ:thm_iteration_2_8}
    \sum_{k=1}^{n_y}\bigg(\frac{\partial g^i}{\partial 
    y^k}(x,\gamma(x))+\eps B^{i,k}\bigg)\cdot\bigg(\frac{\partial
     \Gamma_1^k}{\partial x^j}(x,\eps)-\frac{\partial 
    \gamma^k}{\partial x^j}(x)-\eps\frac{\partial 
    \gamma_1^k}{\partial x^j}(x)\bigg)=\mathcal{O}(\eps^2),
  \end{equation}
  where $\mathcal{O}(\eps^2)$ can be controlled uniformly, 
  $B=B(x)=[B^{i,k}]_{i,k=1}^{n_y}$ defined as
  \begin{displaymath}
    B^{i,k} = \sum_{\ell=1}^{n_y}\frac{\partial^2 
      g^i}{\partial y^k \partial 
      y^\ell}(x,\gamma(x))\gamma_1^\ell(x).
  \end{displaymath}
  \cref{equ:thm_iteration_2_8} can be rewritten as 
  \begin{equation}
    \label{equ:thm_iteration_2_9}
    (G_y(x)+\eps B) (\nabla_x\Gamma_1(x,\eps) - \nabla\gamma(x) - 
    \eps\nabla\gamma_1(x)) = \mathcal{O}(\eps^2).
  \end{equation}
  Notice that $B(x)$ is uniformly bounded thanks to 
  \cref{assump:bound}. 
  Therefore, as long as $\eps$ is sufficiently 
  small, $G_y(x)+\eps B$ is invertible and $(G_y(x)+\eps B)^{-1}$ 
  is uniformly bounded. Therefore, the proof of 
  \cref{thm:iteration_2} is completed by 
  \cref{equ:thm_iteration_2_9}.
  
  By \cref{remark:monotone}, one obtains that 
  \begin{displaymath}
    \begin{aligned}
      & |\Gamma_2(x,\eps) - \gamma(x) - \eps\gamma_1(x) - 
      \eps^2\gamma_2(x)| \\
      \le & \frac{1}{\beta} |g(x,\Gamma_2(x,\eps)) - 
      g(x,\gamma(x)+\eps\gamma_1(x)+\eps^2\gamma_2(x))|\\
      \le & \frac{1}{\beta} |\eps 
      \nabla_x \Gamma_1(x,\eps) f(x,\Gamma_1(x,\eps)) - 
      g(x,\gamma(x)+\eps\gamma_1(x)+\eps^2\gamma_2(x))|.
    \end{aligned}
  \end{displaymath}
  By \cref{thm:iteration_1}, \cref{thm:iteration_2} and Taylor's 
  expansion of $f(x,y)$ at $(x,\gamma(x))$, one gets that 
  \begin{equation}
    \label{equ:thm_iteration_3_2}
    \begin{aligned}
      &\eps\nabla_x\Gamma_1(x,\eps)f(x,\Gamma_1(x,\eps))\\
      =\ 
      &\eps\big(\nabla\gamma(x)+\eps\nabla\gamma_1(x)\big)\big(F(x)+\eps
       F_y(x)\gamma_1(x)\big)+\mathcal{O}(\eps^3)\\
      =\ 
      &\eps\nabla\gamma(x)F(x)+\eps^2\big(\nabla\gamma(x)F_y(x)\gamma_1(x)+\nabla\gamma_1(x)F(x)\big)+\mathcal{O}(\eps^3).
    \end{aligned}
  \end{equation}
  By Taylor's expansion of $g(x,y)$ at $(x,\gamma(x))$, one gets
  that 
  \begin{equation}
    \label{equ:thm_iteration_3_3}
    \begin{aligned}
      &g(x,\gamma(x)+\eps\gamma_1(x)+\eps^2\gamma_2(x))\\
      =\ &\eps 
      G_y(x)\gamma_1(x)+\eps^2\Bigg(G_y(x)\gamma_2(x)+\frac{1}{2}\Big[
       \sum_{j,k=1}^{n_y} G_{yy}^{i,j,k} \gamma_1^j(x) \gamma_1^k(x) 
      \Big]_{i=1}^{n_y}\Bigg)+\mathcal{O}(\eps^3).
    \end{aligned}
  \end{equation}
  By comparing \cref{equ:thm_iteration_3_2} and 
  \cref{equ:thm_iteration_3_3}, one obtains that 
  \begin{displaymath}
    |\eps \nabla_x \Gamma_1(x,\eps) f(x,\Gamma_1(x,\eps)) - 
    g(x,\gamma(x) + \eps\gamma_1(x) + \eps^2\gamma_2(x))| = 
    \mathcal{O}(\eps^3),
  \end{displaymath}
  where $\mathcal{O}(\eps^3)$ can be controlled uniformly thanks to 
  \cref{assump:bound}. In other words, there exists a 
  constant $C_2>0$ such that 
  $|\eps 
  \nabla_x\Gamma_1(x,\eps)f(x,\Gamma_1(x,\eps))-g(x,\gamma(x)+\eps\gamma_1(x)+\eps^2\gamma_2(x))|\le
   C_2\eps^3$. 
  Therefore, 
  \begin{displaymath}
    |\Gamma_2(x,\eps)-\gamma(x)-\eps\gamma_1(x)-\eps^2\gamma_2(x)|\le\frac{C_2}{\beta}\eps^3.
  \end{displaymath}
  Thus, the proof for \cref{thm:iteration_3} is completed.
\end{proof}

\section{Proofs of 
\cref{thm:iteration_attractor,thm:iteration_invariant,thm:nd}}
\label{app:thm_iteration}

\subsection{Several lemmas}
\label{app:thm:lemma}

In the proofs of these three theorems, we have to calculate 
high-order gradients of vector-valued functions. 
Here we present several facts about high-order derivatives.

\begin{lemma}
	\label{lemma:highOrderD}
  Suppose that the functions $A$, $\xi$, $\eta$ in this lemma are 
  sufficiently smooth. We have the following conclusions.
	\begin{enumerate}
    \item \label{lemma:highOrderD_3}
    Assume that $k \ge 0$ and $\xi, \eta \in W^{k,\infty} (\mr^n, 
    \mr)$, then 
    \begin{displaymath}
      |\xi \eta|_{k,\infty} \lesssim \sum_{j = 0}^{k} 
      |\xi|_{j,\infty} |\eta|_{k-j,\infty}.
    \end{displaymath}
    where the bound is dependent on $k$, and independent of $\xi$ and 
    $\eta$.
		
		\item \label{lemma:highOrderD_2} 
    Assume that $k \ge 0$. Let $\xi: \mr^n \rightarrow \mr$ and 
    $\eta: \mr^m \rightarrow \mr^n$ be sufficiently smooth functions. 
    If $\nabla \xi \in W^{k, \infty}$ and $\nabla \eta \in W^{k, 
    \infty}$, then 
    \begin{displaymath}
      |\xi \circ \eta|_{k+1, \infty} \lesssim \sum_{j=1}^{k+1} 
      |\xi|_{j,\infty} \norm{\nabla \eta}_{k+1-j,\infty}^j,
    \end{displaymath}
    where the bound is dependent on $k$, and independent of $\xi$ and 
    $\eta$. In particular, if $\xi \in W^{k+1,\infty}$ and $\nabla 
    \eta \in W^{k,\infty}$, then $\xi \circ \eta \in W^{k+1,\infty}$.
    
    \item \label{lemma:highOrderD_1} 
    Assume that $k \ge 0$. Let $A: \mr^m \rightarrow \mr^{n \times 
    n}$ is a matrix-valued function. If $\nabla A \in W^{k, \infty}$ 
    and $A(\cdot)^{-1} \in \ml^\infty$, then 
    \begin{displaymath}
      |A(\cdot)^{-1}|_{k+1,\infty} \lesssim \sum_{j = 1}^{k+1} 
      \norm{A(\cdot)^{-1}}_{0,\infty}^{j+1} \norm{\nabla 
      A}_{k+1-j,\infty}^j,
    \end{displaymath}
    where the bound is dependent on $k$, and independent of $A$.
  \end{enumerate}
\end{lemma}
\begin{proof}
  \cref{lemma:highOrderD_3} is a direct corollary of Leibniz rule.
  
  As for \cref{lemma:highOrderD_2}, first we notice that 
  \begin{displaymath}
    \nabla_x \left( (\xi \circ \eta)(x) \right) = \left( (\nabla \xi) 
    \circ \eta \right) (x) \cdot \nabla\eta(x),
  \end{displaymath}
  which yields directly that the conclusion holds for $k = 0$. Now we 
  assume that the conclusion holds for all non-negative integers that 
  are less than $k$, and then one can obtain by 
  \cref{lemma:highOrderD_3} and the induction hypothesis that 
  \begin{displaymath}
    \begin{aligned}
      |\xi \circ \eta|_{k+2,\infty} & = \left| \big( (\nabla\xi) 
      \circ \eta \big) \cdot \nabla \eta \right|_{k+1,\infty} \\
      & \lesssim \sum_{i=0}^{k+1} \left|(\nabla\xi) 
      \circ \eta \right|_{i,\infty} |\nabla \eta|_{k+1-i,\infty} \\
      & \lesssim |\xi|_{1,\infty} |\nabla \eta|_{k+1,\infty} + 
      \sum_{i=1}^{k+1} \sum_{j=1}^{i} |\xi|_{j+1,\infty} 
      \norm{\nabla\eta}_{i-j,\infty}^j |\nabla \eta|_{k + 1 - i, 
      \infty} \\
      & \le |\xi|_{1,\infty} |\nabla \eta|_{k+1,\infty} + 
      \sum_{j=1}^{k+1} \sum_{i=j}^{k+1} |\xi|_{j+1,\infty} 
      \norm{\nabla \eta}_{k+1-j,\infty}^{j+1} \\
      & \lesssim \sum_{j = 1}^{k + 2} |\xi|_{j,\infty} \norm{\nabla 
      \eta}_{k + 2 - j, \infty}^j,
    \end{aligned}
  \end{displaymath}
  which completes the proof.
  
  As for \cref{lemma:highOrderD_1}, one can notice first that 
  $A(x)^{-1}$ is continuous, which yields that 
  \begin{displaymath}
    \pd{A(x)^{-1}}{x^i} = -A(x)^{-1} \pd{A(x)}{x^i} A(x)^{-1}, \ 
    \forall i = 1, 2, \ldots, m.
  \end{displaymath}
  Then the conclusion can be proved similarly to 
  \cref{lemma:highOrderD_2}.
\end{proof}


\begin{lemma}
  \label{lemma:taylor}
  Suppose that $\Phi: \mr^n \times [-\delta, \delta] \rightarrow \mr$ 
  is sufficiently smooth. If there exists $k, \ell \in \mathbb{N}$ 
  such that the function $\Phi(x,s)$ satisfies
  \begin{displaymath}
    \dfrac{\partial^i \Phi}{\partial s^i}(x, 0) = 0, \ \forall i = 0, 
    1, \ldots, \ell, 
  \end{displaymath}
  and
  \begin{displaymath}
    \sup_{s \in [-\delta, \delta]} \norm{ \dfrac{\partial^{\ell + 1} 
    \Phi}{\partial s^{\ell + 1}}(\cdot, s) }_{k, \infty} < + \infty,
  \end{displaymath}
  then 
  \begin{displaymath}
    \norm{\Phi(\cdot, h)}_{k, \infty} = \mathcal{O}(|h|^{\ell + 1}), 
    \text{ as } h \rightarrow 0.
  \end{displaymath}
\end{lemma}
\begin{proof}
  By Taylor's expansion, for each $j = 0, 1, \ldots, k$, there exists 
  $\theta \in [0,1]$ such that 
  \begin{displaymath}
    |\nabla_x^j \Phi(x,h)| = \frac{1}{(\ell+1)!} \left| \nabla_x^j 
    \dfrac{\partial^{\ell+1} \Phi}{\partial s^{\ell+1}}(x,\theta h) 
    h^{\ell+1} \right| \lesssim \sup_{s \in [-\delta, \delta]} \norm{ 
    \dfrac{\partial^{\ell + 1} \Phi}{\partial s^{\ell + 1}}(\cdot, s) 
    }_{k, \infty} |h|^{\ell+1},
  \end{displaymath}
  which completes the proof.
\end{proof}

\begin{corollary}
  \label{lemma:diff}
  Suppose that $\xi: \mr^n \rightarrow \mr$ and $\eta: \mr^n 
  \rightarrow \mr^n$ are sufficiently smooth, and $\eta 
  \in W^{k, \infty}$. We have the following conclusions.
  \begin{enumerate}
    
    \item \label{lemma:highOrderD_6}
    If $\nabla^2 \xi \in W^{k, \infty}$, 
    then $$ \norm{\xi(\cdot + \eta(\cdot) h) - \xi(\cdot) - h \nabla 
    \xi(\cdot) \eta(\cdot)}_{k, \infty} = \mathcal{O}(|h|^2), \text{ 
    as } h \rightarrow 0.$$
     
     \item \label{lemma:highOrderD_7}
     If $\nabla^3 \xi \in W^{k, \infty}$, then $$ \norm{\xi(\cdot + 
     \eta(\cdot) h) - 
     \xi(\cdot - \eta(\cdot) h) - 2 h \nabla \xi(\cdot) 
     \eta(\cdot)}_{k, \infty} = \mathcal{O}(|h|^3), \text{ as }
     h \rightarrow 0.$$
  \end{enumerate}
\end{corollary}

\begin{proof}
  Let $\Phi(x,s) = \xi(x + \eta(x) s) - \xi(x) - s \nabla \xi(x) 
  \eta(x)$. By direct calculation, one can obtain that
  \begin{displaymath}
    \begin{aligned}
      & \pd{\Phi}{s}(x,s) = \nabla \xi \big( x + \eta(x) s \big) 
      \eta(x) - \nabla \xi(x) \eta(x), \\
      & \dfrac{\partial^2 \Phi}{\partial s^2}(x, s) = \sum_{i_1, i_2 
      = 1}^{n} \dfrac{\partial^2 \xi}{\partial x^{i_1} \partial 
      x^{i_2}} \big(x + \eta(x) s \big) \eta^{i_1}(x) \eta^{i_2}(x), 
      \\
      & \dfrac{\partial^3 \Phi}{\partial s^3}(x, s) = \sum_{i_1, i_2, 
      i_3 = 1}^{n} \dfrac{\partial^3 \xi}{\partial x^{i_1} \partial 
      x^{i_2} \partial x^{i_3}} \big(x + \eta(x) s \big) 
      \eta^{i_1}(x) \eta^{i_2}(x) \eta^{i_3}(x). \\
    \end{aligned}
  \end{displaymath}
  Notice that $\Phi(x,0) = \pd{\Phi}{s}(x,0) = 0$. By 
  \cref{lemma:highOrderD}, $\Phi(x,s)$ satisfies the conditions of 
  \cref{lemma:taylor} when $\ell = 1$, which yields 
  \cref{lemma:highOrderD_6}. Similarly, $\Phi(x,s) - \Phi(x,-s)$ 
  satisfies \cref{lemma:taylor} when $\ell = 2$, which completes the 
  proof.
\end{proof}

\begin{lemma}
	\label{lemma:basic}
  If $z_0, z_1 \in \mc^k(\rx, \ry)$ for some $k = 1, 2, \ldots, K$, 
  then there exist unique $y_0,y_1\in\mathcal{C}^k(\rx,\ry)$ 
  satisfying that $ g(x,y_i(x)) = z_i(x)$ for $i=0, 1$.
	In addition, if $z_i \in W^{k,\infty}$ for $i=0,1$, then 
	\begin{displaymath}
    \norm{y_1 - y_0}_{k, \infty} \lesssim \norm{z_1 - z_0}_{k, 
    \infty},
	\end{displaymath}
  where the bound is dependent on $\beta$, $k$ and 
  $\norm{z_i}_{k,\infty}$ where $i = 0,1 $.
\end{lemma}
\begin{proof}
	By \cref{remark:monotone} and implicit function 
	theorem for $g(x,y) -s z_1(x) - (1-s) z_0(x) = 0$ with 
	respect to $y$, there exists a unique function 
	$y = y(x,s) \in \mc^k(\rx \times [0,1], \ry)$ such that 
	\begin{equation}
		\label{lemma:basic_1}
		g(x,y(x,s))=s z_1(x)+(1-s)z_0(x).
	\end{equation}

	Now we assume that $z_i \in W^{k,\infty}$ for $i=0,1$.
	Let $y_i(x)=y(x,i)$ for $i = 0, 1$. For each $j = 0, 1, 
	\ldots, k$, there exists $\theta \in [0,1]$ such that 
	\begin{equation}\label{lemma:basic_0}
		|\nabla^j y_1(x) - \nabla^j y_0(x)| \le \left| \pd{}{s} 
		\nabla^j_x y(x,\theta) \right|.
	\end{equation}
	By \cref{lemma:basic_1}, one obtains that
	\begin{equation}
		\label{lemma:basic_2}
		\pd{y}{s} (x,s) = \left( \pd{g}{y} (x,y(x,s)) \right) ^{-1} 
		(z_1(x) - z_0(x))
	\end{equation}
	and
	\begin{equation}
		\label{lemma:basic_3}
		\nabla_x y(x,s) = \left( \pd{g}{y} (x,y(x,s)) \right) ^{-1} 
		\left( s \nabla z_1(x) + (1-s) \nabla z_0(x) - 
		\pd{g}{x}(x,y(x,s)) \right).
	\end{equation}

  We assert that $\sup\limits_{s \in [0,1]}\norm{\nabla_x y(\cdot, 
  s)}_{k-1, \infty}$ is bounded, where the bound is dependent on 
  $\beta$, $k$ and $\norm{z_i}_{k,\infty}$. By \cref{lemma:basic_3}, 
  $\sup\limits_{s \in [0,1]}\norm{\nabla_x y(\cdot, s)}_{0, \infty}$ 
  is bounded. Now assume that $\sup\limits_{s \in 
  [0,1]}\norm{\nabla_x y(\cdot, s)}_{j, \infty}$ is bounded for some 
  $j= 0, 1, \ldots, k-2$. 
  \cref{lemma:highOrderD} yields that $\sup\limits_{s \in 
  [0,1]}\norm{\pd{g}{(x,y)}(\cdot,y(\cdot,s))}_{j+1, \infty}$ and 
  then $\sup\limits_{s \in [0,1]} \norm{\nabla_x y(\cdot, s)}_{j+1, 
  \infty} $ are bounded, which yields the assertion.
  
  By the above assertion and \cref{lemma:highOrderD}, 
  $\left( \pd{g}{y} (x,y(x,s)) \right) ^{-1} $ is bounded in 
  $W^{k,\infty}$, where the bound depends on $\beta$, $k$ and 
  $\norm{z_i}_{k,\infty}$. Therefore, by \cref{lemma:basic_2}, one 
  obtains that 
  \begin{displaymath}
    \sup_{s \in [0,1]} \norm{\pd{y}{s} (\cdot, s)}_{k, \infty} 
    \lesssim \norm{z_1 - z_0}_{k, \infty}.
  \end{displaymath}
  Together with \cref{lemma:basic_0}, the proof is completed.
\end{proof}

\begin{lemma}
	\label{lemma:rhs}
  Assume that $y_i: \rx \rightarrow \ry$ satisfies that $y_i \in 
  \mc^{k+1}$ and $\nabla y_i \in W^{k,\infty}$ for each $i = 0, 1$ 
  and some $k = 0, 1, \ldots, K - 1$. 
  If $y_1 - y_0 \in \ml^\infty$, then 
	\begin{displaymath}
		\norm{\nabla y_1(\cdot) f(\cdot, y_1(\cdot)) - \nabla y_0(\cdot) 
		f(\cdot, y_0(\cdot))}_{k, \infty} \lesssim \norm{y_1 - y_0}_{k+1, 
		\infty},
	\end{displaymath}
  where the bound depends on $f$, $k$ and 
  $\norm{\nabla y_i}_{k,\infty}$ with $i = 0, 1$.
\end{lemma}
\begin{proof}
	Let $\Phi(x,s) = s y_1(x) + (1-s) y_0(x)$
	and $\Psi(x,s) = \nabla_x \Phi(x,s) f(x,\Phi(x,s))$.
	For each $j = 0, 1, \ldots, k$, there exists $\theta \in [0,1]$ 
	such that 
	\begin{displaymath}
		|\nabla_x^j \Psi(x,1) - \nabla_x^j \Psi(x,0)| \le \left| \pd{}{s} 
		\nabla_x^j \Psi(x,\theta) \right|.
	\end{displaymath}
	Actually, 
	\begin{equation}
		\label{equ:lemma_iter_2}
    \begin{aligned}
      \pd{\Psi}{s} (x,s) & = \nabla_x \left( \pd{\Phi}{s} (x,s) 
      \right) f(x,\Phi(x,s)) + \nabla_x \Phi(x,s) 
      \pd{f}{y}(x,\Phi(x,s)) \pd{\Phi}{s}(x,s) \\
      & = (\nabla y_1(x) - \nabla y_0(x)) f(x,\Phi(x,s)) + \nabla_x 
      \Phi(x,s) \pd{f}{y}(x,\Phi(x,s)) (y_1(x) - y_0(x)).
    \end{aligned}
	\end{equation}
  By \cref{lemma:highOrderD}, one obtains that $\sup\limits_{s \in 
  [0,1]} \norm{f(\cdot,\Phi(\cdot,s))}_{k, \infty}$ and 
  $\sup\limits_{s \in [0,1]} \norm{\nabla_x \Phi(\cdot,s) 
  \pd{f}{y}(\cdot,\Phi(\cdot,s))} 
  _{k,\infty}$ are bounded, and the bound depends on $f$, $k$ and 
  $\norm{\nabla y_i}_{k,\infty}$ with $i = 0,1$. Therefore, 
  \begin{displaymath}
    \sup_{s \in [0,1]} \norm{\pd{\Psi}{s}(\cdot, s)}_{k, \infty} 
    \lesssim \norm{y_1 - y_0}_{k+1, \infty},
  \end{displaymath}
	which completes the proof.
\end{proof}

\begin{corollary}
	\label{lemma:iter}
  For $k = 0, 1, \ldots, K$, 
	\begin{equation}
		\label{equ:lemma_iter}
		\norm{ \Gamma_{k+1} (\cdot, \eps) - \Gamma_k (\cdot, \eps)} 
		_{K-k, \infty} = \mathcal{O}(\eps^{k+1}).
	\end{equation}
\end{corollary}

\begin{proof}
	By \cref{equ:nabla_gamma} and \cref{lemma:highOrderD}, one can 
	prove by induction that $\norm{\nabla_x \Gamma_0 (\cdot, \eps)}_{K, 
	\infty}$ is bounded. Therefore, $\norm{\eps \nabla_x 
	\Gamma_0(\cdot,\eps) f(\cdot,\Gamma_0(\cdot,\eps))}_{K,\infty} = 
	\mathcal{O}(\eps)$. By 
	\cref{lemma:basic}, $\norm{\Gamma_1(\cdot,\eps) - 
	\Gamma_0(\cdot, \eps)}_{K, \infty} = \mathcal{O}(\eps)$.
	Now assume that \cref{equ:lemma_iter} holds for $0, 1, \ldots, k$ 
	where $k \le K - 1$. By the induction 
	hypothesis, $\norm{\nabla_x \Gamma_j(\cdot, \eps)}_{K-k-1, \infty}$ 
	are bounded for $j = 0, 1, \ldots, k+1$.
	By \cref{lemma:rhs}, one obtains that 
	\begin{displaymath}
		\norm{ \nabla_x \Gamma_{k+1} (\cdot,\eps) 
		f(\cdot,\Gamma_{k+1}(\cdot,\eps)) - \nabla_x \Gamma_k 
		(\cdot,\eps) f(\cdot,\Gamma_k(\cdot,\eps))} _{K - k -1, \infty} = 
		\mathcal{O}(\eps^{k+1}).
	\end{displaymath}
	By \cref{lemma:basic}, the proof is completed.
\end{proof}

\subsection{Proof of \cref{thm:iteration_attractor}}
\begin{proof}[Proof of \cref{thm:iteration_attractor}]
	We prove this theorem by induction.
	
	When $k = 0$, there exists a constant $C > 0$ dependent on $k$ 
    such that 
	\begin{displaymath}
	\begin{aligned}
	\od{|z_0|^2}{t}
        &= \frac{2}{\eps} \ip{z_0}{g(x,y)-g(x,\gamma(x))}
        -2 \ip{z_0}{\nabla \gamma(x) f(x,y)} \\
	&\le -\frac{2 \beta}{\eps} |z_0|^2 + C |z_0|
    \le -\frac{A}{\eps} |z_0|^2 + \frac{C^2}{4(2\beta - A)} \eps.
	\end{aligned}
	\end{displaymath}
	Here the last inequality is due to Cauchy inequality. By Gronwall 
    inequality, 
	\begin{displaymath}
	|z_0(t)|^2 \le \frac{C^2}{A (2 \beta - A)} \eps^2 
    (1 - e^{-\frac{A}{\eps}t}) + e^{-\frac{A}{\eps}t} |z_0(0)|^2
    \le \frac{C^2}{A (2 \beta - A)} \eps^2 
     + e^{-\frac{A}{\eps}t} |z_0(0)|^2.
	\end{displaymath}
	
	Assuming that the theorem holds for $k-1$, where $k = 1, 2, \ldots, 
	K$. One obtains that
	\begin{equation}
	\label{equ:thm_iter_attr_1}
	\begin{aligned}
    \od{|z_k|^2}{t} &= \frac{2}{\eps} 
    \ip{z_k}{g(x,y)-g(x,\Gamma_k(x,\eps))}\\
	&\quad\quad\quad + 2 \ip{z_k}{\nabla_x \Gamma_{k-1}(x,\eps)  
        f(x,\Gamma_{k-1}(x,\eps)) 
        - \nabla_x \Gamma_k(x,\eps) f(x,y)}\\
	&\le -\frac{2\beta}{\eps} |z_k|^2 + 2 |z_k| \cdot 
        |\nabla_x \Gamma_{k-1}(x,\eps) f(x,\Gamma_{k-1}(x,\eps)) 
        - \nabla_x \Gamma_k(x,\eps) f(x,y)|.
	\end{aligned}
	\end{equation}
	\cref{lemma:iter} and \cref{assump:bound} yield 
    that there exists a constant $C>0$ such that 
	\begin{equation}
	\label{equ:thm_iter_attr_2}
	\begin{aligned}
	&2 |\nabla_x \Gamma_{k-1}(x,\eps) f(x,\Gamma_{k-1}(x,\eps)) - 
    \nabla_x \Gamma_k(x,\eps) f(x,y)|\\
	\le\ & 2 |\nabla_x \Gamma_{k-1}(x,\eps) f(x,\Gamma_{k-1}(x,\eps)) 
    - \nabla_x \Gamma_k(x,\eps) f(x,\Gamma_{k-1}(x,\eps))|\\
	&\quad + 2 |\nabla_x \Gamma_{k}(x,\eps) f(x,\Gamma_{k-1}(x,\eps)) 
    - \nabla_x \Gamma_k(x,\eps) f(x,y)|\\
	\le\ &C\eps^k+C|z_{k-1}|.
	\end{aligned}
	\end{equation}
	Take $\tilde{A}\in(A,2\beta)$. Combing 
    \cref{equ:thm_iter_attr_1} and \cref{equ:thm_iter_attr_2}, one 
    gets by Cauchy inequality that 
	\begin{equation}
	\label{equ:thm_iter_attr_3}
	\begin{aligned}
    \od{|z_k|^2}{t} &\le - \frac{2\beta}{\eps} |z_k|^2 
        + C \eps^k |z_k| + C |z_{k-1}| \cdot  |z_k|\\
	&\le -\frac{\tilde{A}}{\eps} |z_k|^2 +
   \frac{C^2}{2(2\beta-\tilde{A})} \eps^{2k+1} + 
   \frac{C^2}{2(2\beta-\tilde{A})} \eps |z_{k-1}|^2.
	\end{aligned}
	\end{equation}
	The induction hypothesis says that there exists $C_{k-1} > 0$ 
    such that
	\begin{equation}
	\label{equ:thm_iter_attr_4}
	|z_{k-1}(t)|^2 \le C_{k-1} ( \eps^{2k} 
        + e^{-\frac{\tilde{A}}{\eps}t} |z_{k-1}(0)|^2 ).
	\end{equation}
	By \cref{equ:thm_iter_attr_3}, \cref{equ:thm_iter_attr_4} and 
    Gronwall's inequality, one gets that 
    \begin{equation}\label{equ:thm_iter_attr_5}
        \begin{aligned}
            |z_k(t)|^2 & \le \frac{C^2 (1 + C_{k-1})}
                {2 \tilde{A} (2\beta - \tilde{A})}
                (1 - e^{-\frac{\tilde{A}}{\eps} t}) \eps^{2k + 2}
            + e^{-\frac{\tilde{A}}{\eps} t} |z_k(0)|^2
            + \frac{C^2 C_{k - 1}}{2 (2\beta - \tilde{A})}
                t e^{-\frac{\tilde{A}}{\eps} t} \eps |z_{k-1}(0)|^2 \\
            & \le \frac{C^2 (1 + C_{k-1})}
                {2 \tilde{A} (2\beta - \tilde{A})} \eps^{2k + 2}
            + e^{-\frac{A}{\eps} t} |z_k(0)|^2
            + \frac{C^2 C_{k - 1} e^{-1}}{2 (\tilde{A} - A) 
                (2\beta - \tilde{A})} e^{-\frac{A}{\eps} t} \eps^2 
                |z_{k-1}(0)|^2.
        \end{aligned}
    \end{equation}
    By \cref{lemma:iter}, one gets that 
    \begin{equation}\label{equ:thm_iter_attr_6}
        |z_{k-1}(0)|^2 \le 2 |z_k(0)|^2 + 2 |z_k(0) - z_{k-1}(0)| ^2
        \le 2 |z_k(0)|^2 + C \eps^{2k}.
    \end{equation}
    By \cref{equ:thm_iter_attr_5} and \cref{equ:thm_iter_attr_6},
    the theorem holds for $k$. Then the theorem is thus proved.
\end{proof}

\subsection{Proof of \cref{thm:iteration_invariant}}
\begin{proof}[Proof of \cref{thm:iteration_invariant}]
  
  
    Since $\nabla_x f(x,y)$, $\nabla_y f(x,y)$ and $\nabla_x 
    \Gamma_k(x,\eps)$ are all bounded, then there exists a 
    constant $C > 0$ such that
    \begin{displaymath}
        \begin{aligned}
            &\left| \od{(x-X_k)}{t}\right | 
            = |f(x,y) - f(X_k,\Gamma_k(X_k,\eps))|\\
            \le\ &|f(x,z_k + \Gamma_k(x,\eps)) 
            - f(x,\Gamma_k(X_k,\eps))| 
            + |f(x,\Gamma_k(X_k,\eps)) 
            - f(X_k,\Gamma_k(X_k,\eps))|\\
            \le\ & C |z_k| + C |x - X_k|.
        \end{aligned}
    \end{displaymath}
    Then one obtains that 
    \begin{displaymath}
        \frac{1}{2} \od{|x-X_k|^2}{t} = \ip{x-X_k}{\od{(x-X_k)}{t}} 
        \le |x-X_k| \cdot \left| \od{(x-X_k)}{t} \right| 
        \le C |z_k| \cdot |x-X_k| + C |x-X_k|^2.
    \end{displaymath}
    Therefore, by \cref{thm:iteration_attractor} and 
    Cauchy-Schwarz inequality, there exists a constant $C_k > 0$ 
    such that 
    \begin{displaymath}
        \od{|x-X_k|^2}{t} \le C_k (|x-X_k|^2 + \eps^{2k + 2}
        + e^{-\frac{\beta}{\eps}t} |z_k(0)|^2 ).
    \end{displaymath}
    By Gronwall inequality,
    \begin{displaymath}
        |x(t)-X_k(t)|^2 \le e^{C_k t} |x(0) - X_k(0)|^2 
        + \eps^{2k + 2} (e^{C_k t} - 1)
        + C_k \frac{e^{C_k t} - e^{-\frac{\beta}{\eps}t}}
        {\frac{\beta}{\eps} + C_k} |z_k(0)|^2 ,
    \end{displaymath}
    which completes the proof.
\end{proof}

\subsection{Proof of \cref{thm:nd}}
\begin{proof}[Proof of \cref{thm:nd}]
	When $k = 0$, \cref{equ:nd} is straightforward since 
	$\Gamma_0(x,\eps) = \hgd_0(x,\eps)$. Assume that 
	\cref{equ:nd} holds for $0, 1, \ldots, k$ where $k = 0, 
	1, \ldots, \lfloor \frac{K}{2} \rfloor - 1$.
	By \cref{lemma:basic}, it suffices to show that 
	\begin{equation}
		\label{equ:nd_3}
		\norm{ \hgd_k (\cdot + f(\cdot, \hgd_k 
		(\cdot,\eps)) \tau,\eps) - \hgd_k (\cdot,\eps) - \tau 
		\nabla_x \Gamma_k (\cdot,\eps) f(\cdot, \Gamma_k (\cdot,\eps)) } 
		_{K - 2k - 2,\infty} = \mathcal{O}(\tau^2).
	\end{equation}
  By the induction hypothesis, $\norm{\nabla_x \hgd_k (\cdot, 
  \eps) }_{K - 2k - 1, \infty}$ is bounded.
	By \cref{lemma:rhs}, 
	\begin{equation}
		\label{equ:nd_4}
		\norm{ \tau \nabla_x \hgd_k (\cdot,\eps) f(\cdot, 
		\hgd_k(\cdot,\eps) ) - \tau \nabla_x \Gamma_k 
		(\cdot,\eps) f(\cdot, \Gamma_k(x,\eps)) } _{K - 2k - 1, \infty} = 
		\mathcal{O}(\eps \tau^2).
	\end{equation}
  Since $\norm{ \nabla_x^2 \hgd_k (\cdot, \eps)} _{K - 2k - 
  2, \infty}$ is bounded, one obtains by \cref{lemma:diff} that 
	\begin{equation}
		\label{equ:nd_5}
		\norm{  \hgd_k (\cdot + f(\cdot, \hgd_k 
		(\cdot,\eps) ) \tau, \eps) - \hgd_k (\cdot,\eps) - \tau 
		\nabla_x \hgd_k (\cdot,\eps) f(\cdot, \hgd_k 
		(\cdot,\eps) ) }_{K - 2k - 2, \infty} =\mathcal{O}(\tau^2).
	\end{equation}
	By \cref{equ:nd_4} and \cref{equ:nd_5}, one can obtain 
	\cref{equ:nd_3}. Then the proof is completed.
\end{proof}

%% file: ms.bbl
\begin{thebibliography}{10}

\bibitem{Weinan2012The}
{\sc A.~Abdulle, W.~E, B.~Engquist, and E.~Vanden-Eijnden}, {\em The
  heterogeneous multiscale method}, Acta Numerica, 21 (2012), pp.~1--87.

\bibitem{Brayton1972A}
{\sc R.~K. Brayton, F.~G. Gustavson, and G.~D. Hachtel}, {\em A new efficient
  algorithm for solving differential-algebraic systems using implicit backward
  differentiation formulas}, Proceedings of the IEEE, 60 (1972), pp.~98--108.

\bibitem{Brezis2010Functional}
{\sc H.~Brezis}, {\em Functional analysis, {S}obolev spaces and partial
  differential equations}, Springer Science \& Business Media, 2010.

\bibitem{Car1985Unified}
{\sc R.~Car and M.~Parrinello}, {\em Unified approach for molecular dynamics
  and density-functional theory}, Physical Review Letters, 55 (1985), p.~2471.

\bibitem{Carr1991Applications}
{\sc J.~Carr}, {\em Applications of centre manifold theory}, vol.~35, Springer
  Science \& Business Media, 2012.

\bibitem{Chartier2012Higher}
{\sc P.~Chartier, A.~Murua, and J.~M. Sanz-Serna}, {\em Higher-order averaging,
  formal series and numerical integration {I}: {B}-series}, Foundations of
  Computational Mathematics, 10 (2010), pp.~695--727.

\bibitem{Chua1986The}
{\sc L.~Chua, M.~Komuro, and T.~Matsumoto}, {\em The double scroll family},
  IEEE Transactions on Circuits and Systems, 33 (1986), pp.~1072--1118.

\bibitem{Stephen1995Initial}
{\sc S.~M. Cox and A.~J. Roberts}, {\em Initial conditions for models of
  dynamical systems}, Physica D: Nonlinear Phenomena, 85 (1995), pp.~126--141.

\bibitem{Weinan2003Analysis}
{\sc W.~E}, {\em Analysis of the heterogeneous multiscale method for ordinary
  differential equations}, Communications in Mathematical Sciences, 1 (2003),
  pp.~423--436.

\bibitem{Weinan2012Review}
{\sc W.~E}, {\em The heterogeneous multiscale method: A ten-year review}, in
  American Physical Society, 2012.

\bibitem{Engquist2005HMMODEs}
{\sc B.~Engquist and Y.-H. Tsai}, {\em Heterogeneous multiscale methods for
  stiff ordinary differential equations}, Mathematics of Computation, 74
  (2005), pp.~1707--1742.

\bibitem{Eriksson2012Explicit}
{\sc K.~Eriksson, C.~Johnson, and A.~Logg}, {\em Explicit time-stepping for
  stiff {ODE}s}, SIAM Journal on Scientific Computing, 25 (2004),
  pp.~1142--1157.

\bibitem{freiling2002survey}
{\sc G.~Freiling}, {\em A survey of nonsymmetric {R}iccati equations}, Linear
  Algebra and its Applications, 351 (2002), pp.~243--270.

\bibitem{Gear2001Projective}
{\sc C.~W. Gear and I.~G. Kevrekidis}, {\em Projective methods for stiff
  differential equations: problems with gaps in their eigenvalue spectrum},
  SIAM Journal on Scientific Computing, 24 (2003), pp.~1091--1106.

\bibitem{Givon2004Extracting}
{\sc D.~Givon, R.~Kupferman, and A.~Stuart}, {\em Extracting macroscopic
  dynamics: model problems and algorithms}, Nonlinearity, 17 (2004),
  pp.~R55--R127.

\bibitem{Guckenheimer2003The}
{\sc J.~Guckenheimer, K.~Hoffman, and W.~Weckesser}, {\em The forced van der
  {P}ol equation {I}: The slow flow and its bifurcations}, SIAM Journal on
  Applied Dynamical Systems, 2 (2003), pp.~1--35.

\bibitem{Songnian2018An}
{\sc S.~He and Q.-L. Dong}, {\em An existence-uniqueness theorem and
  alternating contraction projection methods for inverse variational
  inequalities}, Journal of Inequalities and Applications,  (2018), pp.~1--19.

\bibitem{Heineken1967On}
{\sc F.~G. Heineken, H.~M. Tsuchiya, and R.~Aris}, {\em On the mathematical
  status of the pseudo-steady state hypothesis of biochemical kinetics},
  Mathematical Biosciences, 1 (1967), pp.~95--113.

\bibitem{Jiang2015A}
{\sc Y.~Jiang, R.~Li, and S.~Wu}, {\em A second order time homogenized model
  for sediment transport}, Multiscale Modeling \& Simulation, 14 (2016),
  pp.~965--996.

\bibitem{Kaps1985Rosenbrock}
{\sc P.~Kaps, S.~W.~H. Poon, and T.~D. Bui}, {\em Rosenbrock methods for stiff
  {ODE}s: A comparison of {R}ichardson extrapolation and embedding technique},
  Computing, 34 (1985), pp.~17--40.

\bibitem{Laskar1994Large}
{\sc J.~Laskar}, {\em Large-scale chaos in the solar system}, Astronomy and
  Astrophysics, 287 (1994), pp.~L9--L12.

\bibitem{legoll2013micro}
{\sc F.~Legoll, T.~Lelievre, and G.~Samaey}, {\em A micro-macro parareal
  algorithm: application to singularly perturbed ordinary differential
  equations}, SIAM Journal on Scientific Computing, 35 (2013),
  pp.~A1951--A1986.

\bibitem{Macnamara2008Multiscale}
{\sc S.~Macnamara, K.~Burrage, and R.~B. Sidje}, {\em Multiscale modeling of
  chemical kinetics via the master equation}, Multiscale Modeling \&
  Simulation, 6 (2008), pp.~1146--1168.

\bibitem{Papanicolaou1976Some}
{\sc G.~C. Papanicolaou}, {\em Some probabilistic problems and methods in
  singular perturbations}, The Rocky Mountain Journal of Mathematics,  (1976),
  pp.~653--674.

\bibitem{pavliotis2008multiscale}
{\sc G.~Pavliotis and A.~Stuart}, {\em Multiscale methods: averaging and
  homogenization}, Springer Science \& Business Media, 2008.

\bibitem{Roberts1989Appropriate}
{\sc A.~J. Roberts}, {\em Appropriate initial conditions for asymptotic
  descriptions of the long term evolution of dynamical systems}, Journal of the
  Australian Mathematical Society, 31 (1989), pp.~48--75.

\bibitem{Ryu2016A}
{\sc E.~K. Ryu and S.~Boyd}, {\em Primer on monotone operator methods}, Applied
  \& Computational Mathematics, 15 (2016), pp.~3--43.

\bibitem{Su1992The}
{\sc W.~C. Su, Z.~Gajic, and X.~M. Shen}, {\em The exact slow-fast
  decomposition of the algebraic {R}icatti equation of singularly perturbed
  systems}, IEEE Transactions on Automatic Control, 37 (1992), pp.~1456--1459.

\bibitem{Verhulst2005Methods}
{\sc F.~Verhulst}, {\em Methods and applications of singular perturbations:
  boundary layers and multiple timescale dynamics}, vol.~50, Springer Science
  \& Business Media, 2005.

\bibitem{Hairer1980Solving}
{\sc G.~Wanner and E.~Hairer}, {\em Solving ordinary differential equations
  {II}}, vol.~375, Springer Berlin Heidelberg, 1996.

\bibitem{Wu}
{\sc S.~Wu}, {\em Multiscale modelling and simulation for the channel
  morphodynamic problems in large timescale}, PhD thesis, Peking University,
  2015.

\bibitem{Zhang1998A}
{\sc Y.~Zhang, T.-S. Lee, and W.~Yang}, {\em A pseudobond approach to combining
  quantum mechanical and molecular mechanical methods}, The Journal of Chemical
  Physics, 110 (1999), pp.~46--54.

\end{thebibliography}
